\documentclass{amsart}
\usepackage{amsfonts}

\usepackage{amssymb}
\usepackage{url}
\usepackage{amscd}
\usepackage{texdraw}

\theoremstyle{plain}

\newtheorem*{Shi}{Theorem}
\newtheorem*{rw}{Relative Schwarz Lemma}
\newtheorem*{Herm}{Theorem A}
\newtheorem*{ZB}{Theorem B}

\newtheorem*{M}{Main Theorem}

\newtheorem{corollary}{Corollary}[section]
\newtheorem{lemma}{Lemma}[section]
\newtheorem{remark}{Remark}
\newtheorem{proposition}{Proposition}[section]
\newtheorem{definition}{Definition}[section]
\input txdtools

\begin{document}
\title[Bounded type Siegel disks of rational maps are quasi-disks]
{All Bounded type Siegel disks of rational maps are quasi-disks}
\author{Gaofei Zhang}

\address{Department of  Mathematics \\ Nanjing University \\
210093, P. R. China}
\email{zhanggf@hotmail.com}

\thanks{}


\maketitle

\begin{abstract}
We prove that every bounded type Siegel disk of a rational map must
be a quasi-disk with at least one critical point on its boundary.
This verifies Douady-Sullivan's conjecture in the case of bounded
type rotation numbers.
\end{abstract}

\section{Introduction}
A Siegel disk of a rational map $f$ is a maximal domain on which $f$
is holomorphically conjugate to an irrational rotation. It was
conjectured by Douady and Sullivan in 1980's that the boundary of
every Siegel disk for a rational map has to be a Jordan curve
\cite{D1}. This has remained an open problem, even for quadratic
polynomials. The main purpose of this paper is to verify this
conjecture under the condition that the rotation number of the
Siegel disk is of bounded type. Here we say an irrational number $0<
\theta < 1$ is of bounded type if $\sup\{a_{k}\} < \infty$ where
$\theta = [a_{1}, \cdots, a_{n}, \cdots,]$ is the continued fraction
of $\theta$. Before we state the main result of the paper, let us
give a brief account of the previous studies on this problem.

In 1986, Douady observed that
quasisymmetric linearization of  critical circle mappings
would imply that the boundary of the Siegel disk of a quadratic polynomial is a quasi-circle.
Using work of Swiatek, Herman then proved the required quasisymmetric linearization result for
analytic  circle mappings with bounded type rotation numbers. This implies that every bounded type Siegel disk of
a quadratic polynomial  must be a
quasi-disk whose boundary passes through the unique finite critical point
of the quadratic polynomial \cite{D2}.  In 1998, by considering a
surgery map defined on certain space of some degree-5 Blaschke
products, Zakeri extended Douady-Herman's result to bounded type Siegel
disks of  all cubic polynomials \cite{Z}. Shortly after that, in his
webpage Shishikura announced
\begin{Shi}[Shishikura]
All bounded type Siegel disks of  polynomial maps are quasi-disks
which have at least one critical point on their boundaries.
\end{Shi}
The main purpose of this paper is to generalize the above result to
bounded type Siegel disks of all rational maps.
\begin{M}
Let $d \ge 2$ be an integer  and $0< \theta < 1$ be an irrational
number of bounded type. Then there exists a constant $1< K(d,
\theta)< \infty$ depending only on $d$ and $\theta$ such that for
any  rational map $f$ of degree $d$, if $f$ has a fixed Siegel disk
with rotation number $\theta$, then   the boundary of the Siegel
disk is a $K(d, \theta)$-quasi-circle which passes through at least
one critical point of $f$.
\end{M}
There are two main ingredients in the proof of the Main Theorem. The
first one is due to Shishikura by which he proved that bounded type
Siegel disks of polynomial maps are all quasi-disks.  The idea of
Shishikura is to prove that any invariant curve inside a bounded
type Siegel disk of a polynomial map is uniformly quasiconformal.
The result then follows by letting the invariant curve approach the
boundary of the Siegel disk.  A detailed description of this
strategy will be given in $\S3$ of this paper.

The second one is an extension of Herman's uniform quasisymmetric
bound to all analytic circle mappings induced by  $\emph{centered}$ Blaschke products (for the definition of $\emph{centered}$ Blaschke products, see $\S2$).
 As indicated by Shishikura, the key tool used in his proof is a
uniform quasisymmetric bound of the linearization maps for a compact
family of analytic circle mappings, which was due to Herman (see
Theorem A of $\S2$). The main obstruction in generalizing
Shishikura's result to all rational maps is that the family of
Blaschke products involved in constructing Siegel disks of rational
maps is not compact anymore, and Herman's theorem does not apply
directly in this situation. The core of our proof is an
extension of Herman's theorem to all centered Blaschke products (see Theorem B of $\S2$). This is the
heart of the whole paper.  One of the key tools used in our proof is the Relative Schwarz Lemma proved by Buff and
Ch\'{e}ritat in \cite{BC}.

The following is a sketch of the organization of the paper.

In $\S2$, we  introduce  Herman's theorem and its extension (Theorem
A and Theorem B).   Since the proof of Theorem B is quite long, we
postpone it until the last section of the paper.

In $\S3$, we prove the Main Theorem. The proof is divided  into two
steps. In the first step,  we prove the Main Theorem under the
condition that  the post-critical set of the rational map does not
intersect the interior of the Siegel disk (Lemma~\ref{con-ext}). In
the second step we prove the Main Theorem in the general case
(Lemma~\ref{post-in}). The proof of Lemma~\ref{con-ext} is based on
Theorem B and Shishikura's strategy. The proof of
Lemma~\ref{post-in} uses Lemma~\ref{con-ext}  and a trick of
holomorphic motion.

In $\S4$, we  prove Theorem B and thus complete the proof of the
Main Theorem.

$\bold{Acknowledgement.}$ The author would like to thank Saeed
Zakeri from whom the author learned the Shishikura's construction of  Siegel disks, which plays a crucial role in the proof of the Main Theorem.  Further
thanks are due to Arnaud Ch$\acute{e}$ritat who carefully read the
manuscript and provided many invaluable comments, especially, he
pointed out to the author that the proof in an early version of the manuscript
does not cover the case that the post-critical set  intersects
the interior of the Siegel disk. Finally, the author would like to express his deep thanks to the two anonymous  referees  for their very detailed reports  which  help the author greatly simplify  and improve  the  original proof.

%%%%%%%%%%%%%%%%%%%%%%%%%%%%%%%%%%%%%%%%%%%%%%%%%%%%%%%%%%%%%%%%%%%
%%%%%%%%%%%%%%%%%%%%%%%%%%%%%%%%%%%%%%%%%%%%%%%%%%%%%%%%%%%%%%%%%%%
%%%%%%%%%%%%%%%%%%%%%%%%%%%%%%%%%%%%%%%%%%%%%%%%%%%%%%%%%%%%%%%%%%%

\section{Herman's Theorem and its extension}
Let  $m = 2d -1$ with $d \ge 2$ being some integer.  Let $\theta =
[a_{1}, \cdots,a_{n}, \cdots]$ be an irrational number with $\sup
\{a_{n}\} < \infty$.  We call such $\theta$ of bounded type.   Let
${\Bbb T}$ denote the unit circle and $R_{\theta}: {\Bbb T} \to
{\Bbb T}$ denote the rigid rotation given by $z \to e^{2 \pi i
\theta} z$. Let ${\mathbf{H}}_{\theta}^{m}$ denote the class of all
the Blaschke products
\begin{equation}\label{Herman-B}
 B(z) = \lambda z^{d} \prod_{i=1}^{d-1} \frac{1 -
\overline{a}_{i}z}{z - a_{i}},
\end{equation} such that
\begin{itemize}
\item[1.] $|a_{i}| < 1$ for all $1 \le i \le d-1$,
\item[2.] $|\lambda| = 1$,
\item[3.] $B|{\Bbb T}: {\Bbb T} \to {\Bbb T}$ is a  circle
homeomorphism of rotation number $\theta$.
\end{itemize}
In one of his three handwritten manuscripts  \cite{H}(see also \cite{C1} and \cite{C2}),  Herman
proved
\begin{Herm}\label{Herman}
Let $m \ge 3$ be an odd integer and $0< \theta< 1 $ be an irrational
number of bounded type. Then there is a constant $1< K(m, \theta)<
\infty$ depending only on $m$ and $\theta$ such that for any $B \in
{\mathbf{H}}_{\theta}^{m}$, there is a  $K(m,
\theta)$-quasi-symmetric homeomorphism $h_{B}$ of the unit circle
such that $B|{\Bbb T} = h_{B}^{-1} \circ R_{\theta}\circ h_{B}$ and
$h_{B}(1) = 1$, where $R_{\theta}: z \mapsto e^{2 \pi i \theta} z$ is the rigid rotation given by $\theta$.
\end{Herm}

The proof of Theorem A in \cite{H}  depends essentially on the fact that
the family  ${\mathbf H}_{\theta}^{d}$ is compact in the following
sense.
\begin{lemma}\label{comp-H}
There is an annular neighborhood $H$ of ${\Bbb T}$, such that \begin{itemize}
\item[1.]  all maps
in ${\mathbf{H}}_{\theta}^{m}$ are holomorphic in $H$, and \item[2.]  for
any sequence $\{B_{n}\} \subset {\mathbf{H}}_{\theta}^{m}$, there is a
subsequence $\{B_{n_{k}}\}$ such that  $B_{n_{k}}|H$ converges uniformly
to  $B|H$ where $B \in {\mathbf{H}}_{\theta}^{l}$ and   $1 \le l \le
m$ is  some odd integer. \end{itemize}
\end{lemma}
\begin{proof}
By  $\S15$ of \cite{H},  there is a $0< \rho < 1$ such that for any
$B \in {\mathbf{H}}_{\theta}^{m}$ given by {\rm (\ref{Herman-B})},
one has $ |a_{i}| \le \rho$.  Let
$$
H = \{z\:|\: (1+\rho)/2 < |z| < 2\}.
$$
Then all the maps in
${\mathbf{H}}_{\theta}^{m}$ are holomorphic in $H$.  This proves the first assertion. Let
$$
 B_{n}(z) = \lambda z^{d} \prod_{i=1}^{d-1} \frac{1 -
\overline{a}_{n,i}z}{z - a_{n,i}}.
$$
Since $|a_{n,i}| \le \rho$, there is a subsequence of integers $\{n_{k}\}$ such that for each $1 \le i \le d-1$,  $a_{n_{k}, i} \to b_{i}$  with $0\le |b_{i}| \le \rho$. It follows that as $k  \to \infty$,
$$
\frac{1 -
\overline{a}_{n_{k},i}z}{z - a_{n_{k},i}} \to \frac{1 -
\overline{b}_{i}z}{z - b_{i}}
$$
uniformly on $H$. Let
$$
 B(z) = \lambda z^{d} \prod_{i=1}^{d-1} \frac{1 -
\overline{b}_{i}z}{z - b_{i}}.
$$
Then $B \in {\mathbf{H}}_{\theta}^{l}$  with  $1 \le l \le
m$ being  some odd integer and  $B_{n_{k}} \to B$ uniformly on $H$. This proves the second assertion and Lemma~\ref{comp-H} follows.

\end{proof}
 Theorem A plays an
important role in the study of bounded type Siegel disks of
polynomial maps. Among all of those the most remarkable one is
Shishikura's result which says that any bounded type Siegel disk of
a polynomial map is a quasi-circle with at least one critical point
on it.

 Let $d$, $m$ and $\theta$
be as above. Let ${\mathbf{B}}_{\theta}^{m}$ denote the class of all
the Blaschke products
\begin{equation}\label{Zhang-B}
B(z) = \lambda \prod_{i=1}^{d} \frac{z - p_{i}}{1-
\overline{p}_{i}z} \prod_{j=1}^{d-1} \frac{z - q_{j}}{1-
\overline{q}_{j}z }
\end{equation}
such that
\begin{itemize}
\item[1.] $|p_{i}| < 1$ and $|q_{j}| > 1$ for all $1 \le i \le d$
and $1 \le j \le d-1$,
\item[2.] $|\lambda| = 1$,
\item[3.] $B|{\Bbb T}: {\Bbb T} \to {\Bbb T}$ is a  circle
homeomorphism of rotation number $\theta$.
\end{itemize}

For any $B \in {\mathbf{B}}_{\theta}^{m}$,  by Herman's result in  \cite{H} it is known that the analytic circle mapping $$B|\Bbb T: \Bbb T \to \Bbb T$$ is quasisymmetrically conjugate to the rigid rotation $R_{\theta}: z \mapsto e^{2 \pi i \theta}$.    Then $B|\Bbb T$ has a unique invariant probability measure on $\Bbb T$ which has  no atoms. Let us denote it by $\mu_{B}$.  According to Douady and Earle \cite{DE}, to such $\mu_{B}$, one can assign a vector field $\xi_{\mu_{B}}$ on $\Delta$ as follows,
$$
\xi_{\mu_{B}}(z) = (1 - |z|^{2}) \int_{\Bbb T} \frac{\zeta - z}{1 - \bar{z} \zeta} d \mu_{B}(\zeta), \:\: z \in \Delta.
$$
By Proposition 1 of \cite{DE}, the vector field $\xi_{\mu_{B}}$ has a unique zero in $\Delta$,  which is called the $\emph{conformal barycenter}$ of $\mu_{B}$. Let us denote it by $z_{B}$.  From the above formula it follows that $z_{B} = 0$ if and only if
\begin{equation}\label{clt}
 \int_{\Bbb T} \:\zeta \:d \mu_{B}(\zeta) = 0.
\end{equation}
Note that for any M\"{o}bius map $g$ which maps the unit circle to itself and preserves the orientation, $g_{*} \mu_{B}$ is the unique invariant probability measure for the analytic circle mapping $(g\circ B \circ g^{-1})|\Bbb T: \Bbb T \to \Bbb T$. It is clear that $g_{*} \mu_{B}$ has no atoms.   According to \cite{DE}, the assignment of $\mu \mapsto \xi_{\mu}$ is conformally natural in the following sense: if $g$ is a M\"{o}bius map which maps the unit circle to itself and preserves the orientation, then
$$
\xi_{g^{*}_{\mu_{B}}}(z) = g'(g^{-1}(z)) \cdot \xi_{\mu_{B}}(g^{-1}(z)).
$$
It follows  that if $g$ maps  $z_{B}$ to $0$,
then  the  conformal barycenter of $g_{*} \mu_{B}$ is $0$.
\begin{definition}{\rm
We say $B$ is a centered  Blaschke product if $z_{B} = 0$.
}
\end{definition}
From the previous observation,  any Blaschke product in ${\mathbf{B}}_{\theta}^{m}$ is conjugate to a centered Blaschke product by a M\"{o}bius map which maps the unit circle to itself and preserves the orientation.
The core of the proof of our Main Theorem is the extension of  Herman's theorem to all the centered Blaschke products in ${\mathbf{B}}_{\theta}^{m}$.

\begin{ZB}\label{main}
Let $m \ge 3$ be an odd integer and $\theta = [a_{1}, \cdots,a_{n},
\cdots]$ be a bounded type irrational number. Then there is a
constant $1< M(m, \theta)< \infty$ depending only on $m$ and
$\theta$ such that for any centered Blaschke  product $B$ in  ${\mathbf{B}}_{\theta}^{m}$, the map
$$h_{B}: {\Bbb T} \to {\Bbb T}$$ is an $M(m,
\theta)$-quasisymmetric homeomorphism, where $h_{B}: {\Bbb T} \to
{\Bbb T}$ is the circle homeomorphism  such that $B|{\Bbb T} = h_{B}^{-1}
\circ R_{\theta}\circ h_{B}$ and $h_{B}(1) = 1$.
\end{ZB}
\begin{remark}\label{rdd}{\rm
We would like to remark that for every odd integer $m \ge 3$ and irrational rotation number  $0< \theta < 1$,  the family of centered Blaschke products in ${\mathbf{B}}_{\theta}^{m}$ is not compact in the sense of Lemma~\ref{comp-H}. One can show that for any annular neighborhood $H$ of the unit circle, there is a centered Blaschke product $B$ in ${\mathbf{B}}_{\theta}^{m}$ such that $B$ is not holomorphic in  $H$.
}
\end{remark}
As an immediate corollary of Theorem B,  we have
\begin{corollary}\label{BC}{\rm
Let $m = 2d -1 \ge 3$ be an odd integer and $\theta = [a_{1}, \cdots,a_{n},
\cdots]$ be a bounded type irrational number. Then there is a
constant $1< K(d, \theta)< \infty$ depending only on $d$ and
$\theta$ such that for any Blaschke product $B$ in  ${\mathbf{B}}_{\theta}^{m}$, the map
$$h_{B}: {\Bbb T} \to {\Bbb T}$$ can be extended to a $K(d, \theta)$-quasiconformal homeomorphism of the unit disk to itself,  where $h_{B}: {\Bbb T} \to
{\Bbb T}$ is the circle homeomorphism  such that  $B|{\Bbb T} = h_{B}^{-1}
\circ R_{\theta}\circ h_{B}$ and $h_{B}(1) = 1$. }
\end{corollary}
%%%%%%%%%%%%%%%%%%%%%%%%%%%%%%%%%%%%%%%%%%%%%%%%%%%%%%%%%%%%%%%%%%%%
%%%%%%%%%%%%%%%%%%%%%%%%%%%%%%%%%%%%%%%%%%%%%%%%%%%%%%%%%%%%%%%%%%%%
%%%%%%%%%%%%%%%%%%%%%%%%%%%%%%%%%%%%%%%%%%%%%%%%%%%%%%%%%%%%%%%%%%%%

\section{Proof of The Main Theorem assuming Theorem B}

Let $d \ge 2$ and $0< \theta< 1$ be an irrational number of bounded
type.  Suppose that $f$ is a rational  map of degree $d$ and has  a
fixed Siegel disk $D$  centered at the origin and with rotation
number $\theta$. By a M\"{o}bius conjugation, we may assume that
$\overline{D}$ is contained in a compact set of the complex plane.
Let $\Delta$ denote the unit disk. Let
$$
\lambda: \Delta \to D
$$
be the holomorphic isomorphism such that $\lambda(0) = 0$,
$\lambda'(0)
> 0$, and
$$
\lambda^{-1} \circ f \circ \lambda(z) = e^{2 \pi i \theta} z
$$
for all $z \in \Delta$. For $0 < r < 1$, let
$$
\Gamma_{r} = \{\lambda(r e^{it})\:\big{|}\: 0 \le t \le 2\pi\}.
$$

%%%%%%%%%%%%%%%%%%%%%%%%%%%%%%%%%%%%%%%%%%%%%%%%%%%%%%%%%%%%%%%%%%%%%
%%%%%%%%%%%%%%%%%%%%%%%%%%%%%%%%%%%%%%%%%%%%%%%%%%%%%%%%%%%%%%%%%%%%%
%%%%%%%%%%%%%%%%%%%%%%%%%%%%%%%%%%%%%%%%%%%%%%%%%%%%%%%%%%%%%%%%%%%%%
Let $K > 1$ and $\widehat{\Bbb C}$ be the Riemann sphere. We call a
simple closed curve $\Gamma \subset \widehat{\Bbb C}$ a
$K$-quasi-circle if there is a $K$-quasiconformal homeomorphism
$$
\phi:\widehat{\Bbb C} \to \widehat{\Bbb C}
$$
such that $\Gamma = \phi(\Bbb T)$ where $\Bbb T$ is the unit circle.

\begin{lemma}\label{inv}
If there exists a $1 < K < \infty$ such that  $\Gamma_{r}$ is a
$K$-quasi-circle for all $0 < r < 1$, then $\partial D$ is a
$K$-quasi-circle. In particular, the map $f|\partial D: \partial D
\to
\partial D$ is injective, and thus $\partial D$ contains at least
one of the critical points of $f$.
\end{lemma}

\begin{proof}
By assumption, for any integer $n > 1$, there is a
$K$-quasiconformal homeomorphism $ \sigma_{n}: \widehat{\Bbb C} \to
\widehat{\Bbb C} $ such that  $$\sigma_{n}(\Bbb T) =
\Gamma_{1-1/n}.$$ We may assume that $\sigma_{n}$ maps the origin
into the inside of $\Gamma_{1-1/n}$. Let $\eta_{n}$ be a M\"{o}bius
map which preserves the unit disk and maps the origin to
$\sigma_{n}^{-1}(0)$. Let $$\omega_{n} = \sigma_{n} \circ
\eta_{n}.$$ Then $\omega_{n}$ is  a $K$-quasiconformal homeomorphism
of the sphere and moreover, $\omega_{n}({\Bbb T}) = \Gamma_{1-1/n}$ and
$\omega_{n}(0) = 0$. It follows that any limit map of the sequence
$\{\omega_{n}\}$ is a $K$-quasiconformal homeomorphism of the
sphere. By taking a convergent subsequence, we may assume that there
is a $K$-quasi-conformal homeomorphism $$\omega: \widehat{\Bbb C} \to \widehat{\Bbb
C}$$ such that $\omega_{n}$ converges uniformly to $\omega$ with
respect to the spherical metric.

We claim that $$D  =
\omega(\Delta).$$ Let us prove the claim now. For $r > 0$, let $\Delta_{r}$ denote the Euclidean disk centered at the origin and with radius $r$. Then for any $1 < n < l$ we have
$\lambda(\Delta_{1-1/n}) \subset \lambda(\Delta_{1-1/l})$. Since $\omega_{l}(\Delta) = \lambda(\Delta_{1-1/l})$, we have
\begin{equation}\label{99g}
\lambda(\Delta_{1-1/n}) \subset \omega_{l}(\Delta).
\end{equation} Let us first prove that
\begin{equation}\label{98g}
\lambda(\Delta_{1-1/n}) \subset \omega(\Delta).
\end{equation}   Suppose (\ref{98g}) were not true. Since $\lambda(\Delta_{1-1/n})$ is open and $\omega(\Delta)$ is a quasi-disk, there would be a point $z \in \lambda(\Delta_{1-1/n})$ such that $d(z, \overline{\omega(\Delta)}) = \delta > 0$. Here $d(\cdot, \cdot)$ denotes the distance with respect to the spherical metric.
Since $\omega_{l} \to \omega$ uniformly with respect to the spherical metric, we have $$d(z, \overline{\omega_{l}(\Delta)}) >  \delta/2 > 0$$ for all $l$ large enough. This is a contradiction with (\ref{99g}). Thus (\ref{98g}) has been proved.
Since $D = \lambda(\Delta)$,  by letting $n \to \infty$ in the left hand of (\ref{98g}), we get
\begin{equation}\label{3g}
D \subset \omega(\Delta).
\end{equation}
Note that for any $l \ge 1$, we have
\begin{equation}\label{6g}
\omega_{l}(\Delta) = \lambda(\Delta_{1-1/l}) \subset \lambda(\Delta) = D.
\end{equation}
For any $z \in \Delta$, let $H = \{\zeta\:|\: |z|< |\zeta| < 1\}$. Since $\omega_{l}(0) = 0$,  $\omega_{l}(H)$ is an annulus contained in $D$ which separates $\{0, \omega_{l}(z)\}$ and $\partial D$.   Since  $\omega_{l}$ is $K$-quasiconformal for all $l$,  it follows that
$$
{\rm mod}(\omega_{l}(H)) \ge \frac{1}{K} {\rm mod}(H) =  \frac{1}{2 K \pi} \log \frac{1}{|z|}.$$  This implies that there  is some $\delta > 0$ independent of $l$ such that $$d(\omega_{l}(z), \partial D) \ge \delta$$ for all $l$. Since $\omega_{l}(z) \in D$, it follows that  $B_{\delta}(w_{l}(z)) \subset D$ for all $l$.   Since  $\omega_{l} \to \omega$ uniformly with respect to the spherical metric, it follows that $w(z) \in D$. Since $z$ is arbitrary, we have
\begin{equation}\label{12g}
\omega(\Delta) \subset D.
\end{equation}
From (\ref{3g}) and (\ref{12g}) it follows that $D = \omega(\Delta)$ and the claim has been proved.

From the claim we have $\partial D =
\omega(\Bbb T)$.
Since $\omega$ is a $K$-quasiconformal homeomorphism of the sphere to itself,  it follows that $\partial D$  is a $K$-quasi-circle and $D$  is a $K$-quasi-disk.  Since $\lambda : \Delta \to D$ is a holomorphic isomorphism,   one can
homeomorphically extended $\lambda$ to $\partial \Delta$. So we have
$$
\lambda^{-1} \circ f \circ \lambda(z) = e^{2 \pi i \theta} z
$$
holds for all $z \in \partial \Delta$. This implies that
$$
f|\partial D: \partial D \to \partial D
$$
is injective. By a result of Herman (see \cite{He}), it follows that
$\partial D$ contains at least one of the critical points of $f$.
This completes the proof of Lemma~\ref{inv}.
\end{proof}

Let $0< r < 1$ and let
$$
D_{r} = \{\lambda(se^{it})\:\big{|}\: 0 \le s < r, \: 0 \le t \le
2\pi\}.
$$
Let
$$
\phi: \widehat{\Bbb C} \setminus \overline{\Delta} \to \widehat{\Bbb
C} \setminus \overline{D_{r}}
$$
be the holomorphic isomorphism such that $\phi(\infty) = \infty$ and
$\phi'(\infty) > 0$.  Take $r < R < 1$. Let
$$
\Theta_{R} = \phi^{-1}(\Gamma_{R}).
$$
Then $\Theta_{R}$ is a real-analytic simple closed curve which
surrounds the closed unit disk. Let $\Theta_{R}^{*}$ denote the
symmetric image of $\Theta_{R}$ about the unit circle. Let
$\Lambda_{R}$ denote the bounded component of $\widehat{\Bbb C}
\setminus \Theta_{R}^{*}$. It is clear that $\Lambda_{R}$ is a
Jordan domain  with smooth boundary which lies in the inside of the unit disk and  contains the origin. Let
$$
A_{R} = \Delta \setminus \overline{\Lambda_{R}}
$$
be the annulus bounded by $\Bbb T$ and $\Theta_{R}^{*}$.

Take $r_{0} > 0$ small enough such that $\overline{\Delta_{r_{0}}} \subset D_{r}$ where $\Delta_{r_{0}}   = \{z\:|\: |z| < r_{0}\}$.  Let $\eta: \Lambda_{R} \to \Delta_{r_{0}}$ be the Riemann isomorphism such that $\eta(0) = 0$ and $\eta'(0) > 0$. Since $\partial \Lambda_{R}$, $\partial \Delta_{r_{0}}$, $\partial \Delta$ and $\partial D_{r}$ are all smooth curves, there is a quasiconformal homeomorphism $\Phi: \widehat{\Bbb C} \to
\widehat{\Bbb C}$ such that \begin{itemize}
\item[1.]  $\Phi(z) = \phi(z)$ in the outside of
the unit disk, and  \item[2.] $\Phi(z) = \eta(z)$ in $\Lambda_{R}$, and  \item[3.] $\Phi$ is
quasiconformal in $A_{R}$. \end{itemize}

 For $\zeta \in \Bbb C \cup \{\infty\}$, let $\zeta^{*} = 1/\bar{\zeta}$ be the symmetric image of $\zeta$ about the unit circle.  Define
\begin{equation}
G(z) =
\begin{cases}
\Phi^{-1} \circ f \circ \Phi(z) & \text{for $|z| \ge 1$,}\\
(\Phi^{-1} \circ f \circ \Phi(z^{*}))^{*} & \text{for $|z|< 1$ }.
\end{cases}
\end{equation}
Let
$$
H_{r} = \Phi^{-1}(\{\lambda(s e^{it})\:\big{|}\: r \le s < 1, 0 \le
t \le 2\pi\}.
$$
Let $H_{r}^{*}$ denote the symmetric image of $H_{r}$ about the unit
circle. Then $H_{r} \cup H_{r}^{*}$ is an annular neighborhood of
the unit circle.  Throughout the following, let us set $$m = 2d - 1.$$ By the construction, we have
\begin{lemma}\label{rr-r}
The map $G$ is a degree $m$ branched covering map of
the sphere to itself which is  holomorphic in $H_{r} \cup
H_{r}^{*}$. Moreover, $G$ is holomorphically conjugate to the rigid
rotation $z \mapsto e^{2 \pi i \theta} z$ in $H_{r} \cup H_{r}^{*}$.
\end{lemma}

From the construction we see if $G$ is quasiconformal at some point
$z$, then $G(z)$ lies in $H_{r} \cup H_{r}^{*}$. Let $\mu_{0}$
denote the standard complex structure in $H_{r} \cup H_{r}^{*}$. Let $G_{0} = G|(H_{r} \cup H_{r}^{*})$ denote the restriction of $G$ to $H_{r} \cup H_{r}^{*}$. By
Lemma~\ref{rr-r}  the map $$G_{0}: H_{r} \cup H_{r}^{*} \to H_{r}\cup H_{r}^{*}$$ is a holomorphic isomorphism and $\mu_{0}$ is $G_{0}$-invariant.  So one can pull back
$\mu_{0}$ by the iteration of $G$ to get a $G$-invariant complex
structure $\mu$ on the whole sphere $\widehat{\Bbb C}$. It follows
from the symmetric property of $G$ and $\mu_{0}$ that $\mu$ is
symmetric about the unit circle.

Note that  if $G$ is quasiconformal at some point $z$ with $|z| >
1$, then $G(z)$ actually belongs to $H_{r}^{*}$ which is contained
in the inside of the unit disk. This implies
\begin{lemma}\label{ee-e}
For almost every $z$ in the outside of the unit disk, if $\mu(z) \ne
0$, then  there exists some integer $k \ge 1$ such that $G^{k}(z)
\in \Delta$.
\end{lemma}

Let $\Psi$ denote the quasiconformal homeomorphism which solves the
Beltrami equation given by $\mu$ and  fixes $0$, $1$, and the
infinity. Let
$$
B(z) = \Psi \circ G \circ \Psi^{-1}(z).
$$
Since $\mu$ is symmetric about the unit circle, the map $$z \mapsto (\Psi(z^{*}))^{*}$$ is also a quasiconformal homeomorphism of the sphere to itself which has  Beltrami coefficient $\mu$. Note that it also fixes $0$, $1$ and the infinity. So $\Psi(z) = (\Psi(z^{*}))^{*}$ for all $z \in \widehat{\Bbb C}$. Since $G(z^{*}) = (G(z))^{*}$ for all $z \in \widehat{\Bbb C}$, it follows that $B(z^{*}) = (B(z))^{*}$ for all $z \in \widehat{\Bbb C}$.   This implies that
\begin{lemma}\label{Shi-con}
$B \in {\mathbf B}_{\theta}^{m}$.
\end{lemma}

\begin{lemma}\label{jj-j}
For almost every $z$ in the outside of the unit disk, if $\Psi^{-1}$
is not conformal at $z$ then there is some integer $k \ge 1$ such
that $B^{k}(z) \in \Delta$.
\end{lemma}
\begin{proof}
Note that $\Psi^{-1}$ is not conformal at $z$ for some $|z| > 1$ if
and only if $\Psi$ is not conformal at $\Psi^{-1}(z)$. From
Lemma~\ref{ee-e} it follows that there is some integer $k \ge 1$
such that $G^{k}(\Psi^{-1}(z)) \in \Delta$. By the symmetric property of $\Psi$, $\Psi$ preserves the unit circle and thus  maps the
unit disk homeomorphically onto the unit disk.  We thus have
$$B^{k}(z) = (\Psi \circ G^{k} \circ \Psi^{-1})(z) \in \Delta.$$ The
lemma follows.
\end{proof}

Let  $h_{B}: {\Bbb T} \to {\Bbb T}$ be the circle homeomorphism such
that $h_{B}(1) = 1$ and
$$
B|{\Bbb T} = h_{B}^{-1} \circ R_{\theta} \circ h_{B}.
$$
By Corollary~\ref{BC}, one can extend $h_{B}$ to a $K(d,
\theta)$-quasiconformal homeomorphism
$$
H_{B}: \Delta \to \Delta
$$
where $1< K(d, \theta)< \infty$ is some constant depending only on
 $d$ and $\theta$. Now let us define
the modified Blaschke product as follows.
$$
\widehat{B}(z) =
\begin{cases}
 B(z) & \text{ for $|z| \ge 1$}, \\
 H_{B}^{-1} \circ R_{\theta} \circ H_{B}(z)& \text{ for $z \in \Delta$}.
\end{cases}
$$

From the above construction,  we have
\begin{proposition}\label{k90}{\rm
Let $z \in \widehat{\Bbb C} \setminus \Delta$. Then $$\widehat{B}(z) \notin \Delta   \Leftrightarrow f(\Phi \circ \Psi^{-1}(z)) \notin D_{r}.$$  Moreover, if $\widehat{B}(z) \notin \Delta$, then
$$
\Phi\circ \Psi^{-1} (\widehat{B}(z)) = f(\Phi \circ \Psi^{-1}(z)).
$$
}
\end{proposition}
%%%%%%%%%%%%%%%%%%%%%%%%%%%%%%%%%%%%%%%%%%%%%%%%%%%%%%%%%%%%%%%%%%%%%%%%
%%%%%%%%%%%%%%%%%%%%%%%%%%%%%%%%%%%%%%%%%%%%%%%%%%%%%%%%%%%%%%%%%%%%%%%%
%%%%%%%%%%%%%%%%%%%%%%%%%%%%%%%%%%%%%%%%%%%%%%%%%%%%%%%%%%%%%%%%%%%%%%%%
Let $\Omega_{\widehat{B}}$ and $\Omega_{f}$ denote critical sets of $\widehat{B}$ and $f$, respectively.
Let $$P_{\widehat{B}} = \bigcup_{k\ge 1}^{\infty} \Omega_{\widehat{B}} \hbox{ and }  P_{f} = \bigcup_{k\ge 1}^{\infty} \Omega_{f}$$ denote the
post-critical sets  of $\widehat{B}$ and $f$, respectively.

\begin{lemma}\label{con-ext}
There is a constant $1 < K(d,
\theta)<\infty$ depending only on
$d$ and $\theta$ such that for any $0< r < 1$,  if  $P_{f} \cap D_{r} = \emptyset$,   then $\Gamma_{r}$ is a $K(d, \theta)$-quasi-circle.
\end{lemma}
\begin{proof}
It suffices to prove the lemma under the stronger assumption that $P_{f} \cap \overline{D_{r}} = \emptyset$. This is because $P_{f} \cap \overline{ D_{r'}} \subset P_{f} \cap D_{r} = \emptyset$ for all $0< r' < r$, and  by the same reasoning as in the proof of Lemma~\ref{inv}, one can show that $\Gamma_{r}$ must be  a $K(d, \theta)$-quasi-circle if $\Gamma_{r'}$ is a $K(d, \theta)$-quasi-circle for all $0< r' < r$.

 Let $\lambda_{r}: \Delta \to D_{r}$ be the holomorphic isomorphism
such that $$\lambda_{r}(1) =  \Phi \circ \Psi^{-1}(1)$$ and
$$
\lambda_{r}^{-1} \circ f \circ \lambda_{r}(z) = e^{2 \pi i \theta} z
$$
for all $z \in \Delta$. Define a quasiconformal homeomorphism
$\chi_{0}: \widehat{\Bbb C} \to \widehat{\Bbb C}$ by
$$
\chi_{0}(z) =
\begin{cases}
 \Phi \circ \Psi^{-1}(z) & \text{ for $|z| \ge 1$}, \\
 \lambda_{r} \circ H_{B}(z)& \text{ for $z \in \Delta$}.
\end{cases}
$$
Let $\mu_{0}$ denote the complex dilatation of $\chi_{0}$ and let
$$
M = \frac{1 + \|\mu_{0}\|_{\infty}}{1 + \|\mu_{0}\|_{\infty}}.
$$ Then  $\chi_{0}$ is an $M$-quasiconformal homeomorphism of
the sphere which maps the unit disk homeomorphically onto $D_{r}$.
Note that $M$ depends on $r$ and may go to infinity as $r \to 1$.

Now for every $k \ge 1$, we will define an $M$-quasiconformal  homeomorphism $\chi_{k}: \widehat{\Bbb C} \to
\widehat{\Bbb C}$ as follows.
Note that  $P_{f} \cap \overline{D_{r}}  = \emptyset$ by the assumption in the beginning of the proof. Since $\Phi\circ \Psi^{-1}$ is a bijection between $\Omega_{\widehat{B}}$ and $\Omega_{f}$,  from Proposition~\ref{k90} it follows that $P_{\widehat{B}}\cap \overline{\Delta} = \emptyset$  and thus  $$\Phi \circ \Psi^{-1}(P_{\widehat{B}}) = P_{f}.$$    So for every $k \ge 1$,  if an  inverse branch of $\widehat{B}^{k}$ maps $\Delta$ to some domain in the outside of the unit disk, then this inverse branch  is univalently defined in an open neighborhood of the closed unit disk. This implies that  each component of $\widehat{B}^{-k}(\Delta)$ is a Jordan domain with  boundary being real analytic,  and moreover, the closures of these Jordan domains are disjoint with each other.

Suppose  $\widehat{B}^{-k}(\Delta)$ has $l_{k}$ components with $l_{k} \ge 1$ being some integer.   Let $$U_{i}, \:\:1 \le i \le l_{k}$$ denote all the components of $\widehat{B}^{-k}(\Delta)$. By the construction of $\widehat{B}$, it follows that $$\Phi \circ \Psi^{-1} (U_{i}),\:\: 1 \le i \le l_{k}$$ are all the components of $f^{-k}(D_{r})$.  Let us first define $\chi_{k}$ on each $U_{i}$.

If $U_{i} = \Delta$, define $$\chi_{k}|\Delta = \lambda_{r} \circ H_{B}.$$
Otherwise, there is a least integer $1 \le k_{0} \le k$ such that  $\widehat{B}^{k_{0}} (U_{i}) = \Delta$. Since $P_{\widehat{B}} \cap \overline{\Delta} = P_{f} \cap \overline{D_{r}} = \emptyset$, the two maps
$$
\widehat{B}^{k_{0}}: U_{i} \to \Delta
$$
and
$$
f^{k_{0}}: \Phi \circ \Psi^{-1} (U_{i}) \to D_{r}
$$
are both holomorphic isomorphisms. So one can lift the quasiconformal homeomorphism $\lambda_{r} \circ H_{B}: \Delta \to D_{r}$ to a quasiconformal homeomorphism $\tau_{i}: U_{i} \to \Phi \circ \Psi^{-1} (U_{i})$ such that the following diagram commutes.
$$
     \begin{CD}
            U_{i}         @  > \tau_{i}   >  >\Phi \circ \Psi^{-1} (U_{i})         \\
           @V \widehat{B}^{k_{0}} VV                         @VV f^{k_{0}} V\\
            \Delta        @  > \lambda_{r} \circ H_{B} >  >         D_{r}
     \end{CD}
$$
In particular, the dilatation of $\tau_{i}$ on $U_{i}$ is equal to that of $\lambda_{r} \circ H_{B}$ on $\Delta$.

Since both $\partial U_{i}$ and $\Phi \circ \Psi^{-1} (\partial U_{i})$ are quasi-circles (in fact, both of them are real analytic curves), $\tau_{i}$ can be homeomorphically extended to $\partial U_{i}$.
Note that $\widehat{B}^{k_{0}}(\partial U_{i}) = \partial \Delta \subset \widehat{\Bbb C} \setminus \Delta$ and thus $\widehat{B}^{k}(\partial U_{i}) \subset \widehat{\Bbb C} \setminus \Delta$ for all $k \ge 0$, by Proposition~\ref{k90}, the following diagram commutes.
$$
     \begin{CD}
            \partial U_{i}         @  > \Phi \circ \Psi^{-1}    >  >\Phi \circ \Psi^{-1} (\partial U_{i})         \\
           @V \widehat{B}^{k_{0}} VV                         @VV f^{k_{0}} V\\
            \partial \Delta        @  > \Phi \circ \Psi^{-1}  >  >        \partial  D_{r}
     \end{CD}
$$
 Since $\Phi \circ \Psi^{-1}|\partial \Delta =
\lambda_{r} \circ H_{B}|\partial \Delta$, from the above two diagrams it follows that $\tau_{i}|\partial U_{i} = \Phi \circ \Psi^{-1}|\partial U_{i}$.
For each  such $U_{i}$, define $\chi_{k} = \tau_{i}$ on  $U_{i}$.

Finally let us define  $\chi_{k} = \Phi \circ \Psi^{-1}$ on the complement of $\widehat{B}^{-k}(\Delta)$.  Since all $\partial U_{i}$, $1 \le i \le l_{k}$,  are quasi-circles which are disjoint with each other,   $\chi_{k}$ is a quasiconformal homeomorphism of the sphere to itself. In this way we get a sequence of quasiconformal homeomorphisms $\chi_{k}: \widehat{\Bbb C} \to \widehat{\Bbb C}$, $k \ge 0$.   We claim
 \begin{itemize}
\item[1.] $\chi_{k}(\Delta) = D_{r}$,
\item[2.]  $\chi_{k}$ is an $M$-quasiconformal homeomorphism of the sphere to itself, and
 \item[3.]  the following diagram commutes. \begin{equation}\label{s-10}
     \begin{CD}
           \widehat{\Bbb C}       @  > \chi_{k+1}   >  > \widehat{\Bbb C}          \\
           @V \widehat{B} VV                         @VV f V\\
             \widehat{\Bbb C}        @  > \chi_{k}  >  >         \widehat{\Bbb C}
     \end{CD}
\end{equation}
\end{itemize}

Let us prove the claim now. The first assertion is obvious since by the construction of $\chi_{k}$,
$\chi_{k}|\Delta = \lambda_{r} \circ H_{B}$ for all $k \ge 0$. Again by the construction of $\chi_{k}$, the dilatation of $\chi_{k}$ on $\widehat{B}^{-k}(\Delta)$ is not greater than the dilatation of $\lambda_{r} \circ H_{B}$ on $\Delta$, and the dilatation of $\chi_{k}$ on $\widehat{\Bbb C} \setminus \widehat{B}^{-k}(\Delta)$ is not greater than the dilatation of $\Phi\circ \Psi^{-1}$ on $\widehat{\Bbb C} \setminus \widehat{B}^{-k}(\Delta)$. So for every $k \ge 1$, the dilatation of $\chi_{k}$ is not greater than the dilatation of $\chi_{0}$ which is  $M$-quasiconformal. The second assertion then follows.  By the construction of $\chi_{k}$ and $\chi_{k+1}$, the following diagram commutes.
$$
     \begin{CD}
           \widehat{B}^{-(k+1)}(\Delta)       @  > \chi_{k+1}   >  > f^{-(k+1)}(D_{r})        \\
           @V \widehat{B} VV                         @VV f V\\
             \widehat{B}^{-k}(\Delta)         @  > \chi_{k}  >  >        f^{-k}(D_{r})
     \end{CD}
$$
 By Proposition~\ref{k90} the following diagram commutes.
$$
     \begin{CD}
           \widehat{\Bbb C} \setminus \widehat{B}^{-(k+1)}(\Delta)       @  > \Phi\circ \Psi^{-1}   >  > \widehat{\Bbb C} \setminus f^{-(k+1)}(D_{r})        \\
           @V \widehat{B} VV                         @VV f V\\
            \widehat{\Bbb C} \setminus \widehat{B}^{-k}(\Delta)         @  > \Phi\circ \Psi^{-1} >  >        \widehat{\Bbb C} \setminus f^{-k}(D_{r})
     \end{CD}
$$ But  on $\widehat{\Bbb C} \setminus \widehat{B}^{-(k+1)}(\Delta)$, $\chi_{k+1} = \Phi\circ \Psi^{-1}$ and on
$\widehat{\Bbb C} \setminus \widehat{B}^{-k}(\Delta)$, $\chi_{k} = \Phi\circ \Psi^{-1}$. This, together with the above two diagrams, implies the third assertion. The claim has been proved.

Now for $k \ge 0$,  let $\mu_{k}$ denote the Beltrami coefficient of
$\chi_{k}$. It follows that
\begin{equation}\label{uni-cc}
\|\mu_{k}\|_{\infty} \le \frac{M-1}{M+1} \end{equation}
 holds for
all $k \ge 0$.

Now let $\nu$ denote the complex dilatation of $\lambda_{r} \circ
H_{B}$ which is defined in the inside of the unit disk. Since $\lambda_{r}$ is conformal in $\Delta = H_{B}(\Delta)$, it follows
that  $\nu$ is equal to the complex dilatation of $H_{B}$. So $\nu$ is $\widehat{B}$-invariant. Since $H_{B}$ is $K(d,
\theta)$-quasiconformal,  we have
$$\|\nu\|_{\infty} \le \frac{K(d, \theta) -1}{K(d, \theta) + 1}.$$

Now let $\Sigma \subset \widehat{\Bbb C}\setminus \Delta$ be the set consisting of all the points $z$ such that $\widehat{B}^{k}(z) \notin \Delta$ for all $k \ge 1$. By Lemma~\ref{jj-j}, it follows
that for almost every $z \in \Sigma$,   $\Psi^{-1}$ is conformal at $z$. Since $\Psi^{-1}(\widehat{\Bbb C} \setminus \Delta) = \widehat{\Bbb C}\setminus \Delta$, and since $\Phi$ is conformal in the outside of the unit disk, it follows that  $\Phi\circ \Psi^{-1}$ is conformal at almost every $z\in \Sigma$.  Now let $$\Xi = \bigcup_{l=0}^{\infty} \widehat{B}^{-l}(\partial \Delta).$$  Then $\Xi$ is the union of countably many real analytic curves and thus is a zero measure set. It is easy to see that for every $z \in \Sigma \setminus \Xi$ and every $k \ge 0$, there is an open neighborhood of such $z$, say $B_{z}(r)$, such that $B_{z}(r) \cap \widehat{B}^{-k}(\Delta) = \emptyset$. By the construction of $\chi_{k}$, it follows that $$\chi_{k}|B_{z}(r) = \Phi\circ \Psi^{-1}|B_{z}(r).$$
This implies that the complex dilatation of $\chi_{k}$ is equal to that of $\Phi\circ \Psi^{-1}$ at $z$. In particular, this implies that for almost every $z \in \Sigma$, $\mu_{k}(z) = 0$ for all $k \ge 0$.

Now suppose $z \in \Sigma$. Then there is
some integer $N \ge 1$ such that $$\widehat{B}^{N}(z) \in \Delta.$$ By the construction of the maps $\{\chi_{k}\}$, $\mu_{N}(z)$ is the pull back of
$\nu(\widehat{B}^{N}(z))$ by $\widehat{B}^{N}$, and $\mu_{k}(z) =
\mu_{N}(z)$ for all $k > N$. Since $\widehat{B}$ is holomorphic  in the outside of the unit disk, we thus have  for all $k \ge N$,
$$
|\mu_{k}(z)| = |\mu_{N}(z)| = |\nu(\widehat{B}^{N}(z))| \le \frac{K(d, \theta) -1}{K(d, \theta)
+ 1}.
$$

Now let us define a Beltrami coefficient $\mu(z)$ on the whole
Riemann sphere by setting
$$
\mu(z) = 0
$$
if $z \in \Sigma$ and
$$
\mu(z) = \mu_{N}(z)
$$
if $\widehat{B}^{N}(z) \in \Delta$ for some $N \ge 0$. It follows
that $$\|\mu\|_{\infty} \le \frac{K(d, \theta) -1}{K(d, \theta) +
1}$$ and $\mu_{k}(z) \to \mu(z)$ for almost every $z \in
\widehat{\Bbb C}$. Now from (\ref{uni-cc}) and the fact that $\chi_{k}|\Delta = \lambda_{r} \circ H_{B}$ for all $k \ge 0$, it follows  that there is a
$K(d, \theta)$-quasiconformal homeomorphism $\chi: \widehat{\Bbb C}
\to \widehat{\Bbb C}$ such that  $\chi_{k}$ converges uniformly to $\chi$ with
respect to the spherical metric.  In particular, we have
$$
\chi|\Delta = \chi_{k}|\Delta = \lambda_{r} \circ H_{B}
$$
for all $k \ge 0$. Note that the quasiconformal homeomorphism $\lambda_{r}\circ H_{B}: \Delta \to D_{r}$ can be homeomorphically extended to $\partial \Delta$ such that $(\lambda_{r} \circ H_{B})(\partial \Delta) = \Gamma_{r}$.
Since $\chi:\widehat{\Bbb C} \to \widehat{\Bbb C}$ is a $K(d, \theta)$-quasiconformal homeomorphism, it follows  that
$\Gamma_{r} = \chi(\partial \Delta)$ is a $K(d, \theta)$-quasi-circle. This completes the
proof  of Lemma~\ref{con-ext}.

\end{proof}

To remove the condition that $P_{f} \cap D_{r} =\emptyset$ in
Lemma~\ref{con-ext}, we need the following lemma.
\begin{lemma}[Lemma 9.8 of \cite{Po}]\label{p}
For any $C > 0$, there is a $1 < K(C) <\infty$ depending only on $C$
such that for any simple closed curve $\gamma \subset \widehat{\Bbb
C}$ if
\begin{equation}\label{cd-e}
\bigg{|}\frac{(w_{1}-w_{3})(w_{2}-w_{4})}{(w_{1}-w_{4})(w_{2}-w_{3})}\bigg{|}
> C
\end{equation}
holds for any four points $\{w_{1}, w_{2}, w_{3}, w_{4}\}$ in
$\gamma$ which are listed according to anticlockwise order, then
$\gamma$ is a $K(C)$-quasi-circle. The converse is also true. That
is, for any $1 < K < \infty$, there exists a $C(K) > 0$ depending
only on $K$ such that if $\gamma\subset \widehat{\Bbb C}$ is a
$K$-quasi-circle, then for any four points $\{w_{1}, w_{2}, w_{3},
w_{4}\}$ in $\gamma$ which are listed according to anticlockwise
order, (\ref{cd-e}) holds with the constant in the right hand
replaced by $C(K)$.
\end{lemma}

Let $R_{\theta}^{d}$ denote the set of all the degree $d$ rational
maps which have a fixed Siegel disk centered at the origin and with
rotation number $\theta$.

\begin{lemma} \label{post-in}
 There is a  $0< C< \infty$ depending only on $d$ and $\theta$ such that
 for any $f \in R_{\theta}^{d}$,  any $0< r < 1$,  any
four distinct integers $k, l, m $ and $n$, and any $z \in
\Gamma_{r}$, if $f^{k}(z), f^{l}(z), f^{m}(z)$ and $f^{n}(z)$ are
ordered anticlockwise in $\Gamma_{r}$, then
$$
\bigg{|}\frac{(f^{k}(z)-f^{m}(z))(f^{l}(z)-f^{n}(z))}
{(f^{k}(z)-f^{n}(z))(f^{l}(z)-f^{m}(z))}\bigg{|}
> C.
$$
\end{lemma}
\begin{proof}

Let  $f \in R_{\theta}^{d}$ and $D$ be the Siegel disk of $f$ centered at the origin. Let $w \in \widehat{\Bbb C} \setminus D$. By considering the rational map $\frac{f(z)}{f(z)-w}$,  we may assume that $\infty \notin D$. For $0 < r < 1$,
let
$$
V_{f}(r;k,l,m,n) = \inf_{z\in\Gamma_{r}}\bigg{|}
\frac{(f^{k}(z)-f^{m}(z))(f^{l}(z)-f^{n}(z))}
{(f^{k}(z)-f^{n}(z))(f^{l}(z)-f^{m}(z))} \bigg{|}
$$

Note that as $z \to 0$, the function
$$ C_{f;k,l,m,n}(z) =  \frac{(f^{k}(z)-f^{m}(z))(f^{l}(z)-f^{n}(z))}
{(f^{k}(z)-f^{n}(z))(f^{l}(z)-f^{m}(z))}
$$ has a non-zero limit.  We can thus regard $C_{f;k,l,m,n}$ as a holomorphic function in
$D$ which does not vanish. In particular, we have
$$V_{f}(r_{1};k,l,m,n) \ge V_{f}(r_{2};k,l,m,n)$$ for all $0 \le
r_{1} < r_{2} < 1$. It is important to note that $V_{f}(r;k,l,m,n)$
is preserved by an M\"{o}bius conjugation.

Let $\{f_{i}\}$ be a sequence in $R_{\theta}^{d}$ such that
$$
\lim_{i \to \infty} V_{f_{i}}(r;k,l,m,n) = \inf_{f \in
R_{\theta}^{d}}\{V_{f}(r;k,l,m,n)\}.
$$
Let $D_{i}$ denote the Siegel disk of $f_{i}$ centered at the
origin.  Let $\Gamma_{r}^{i}$ denote the $\Gamma_{r}$ of $D_{i}$. For each $i$, take $a_{i} \in \Gamma_{r}^{i}$. By
considering the sequence of rational maps $\frac{1}{a_{i}} f_{i} (a_{i}z)$ if necessary,   we may assume that
every  $\Gamma_{r}^{i}$   passes through the point $1$.

For each $i$, let $\phi_{i}: \Delta \to D_{i}$ be the linearization map such that
$\phi_{i}'(0) > 0$. Since every $\Gamma_{r}^{i}$ passes through $1$, it
follows that $\phi_{i}'(0)$ is bounded away from $0$ and the
infinity. By Koebe's $1/4$-Theorem, $D_{i} = \phi_{i}(\Delta)$
contains a Euclidean disk $B_{0}(\tau)$ for some $\tau > 0$. Since $f_{i}(0) = 0$ and $f_{i}'(0) = e^{2 \pi i \theta}$, it follows that the sequence  $\{f_{i}\}$ is normal in $B_{0}(\tau)$. By taking a convergent subsequence, we may assume that $f_{i}$ converges to a univalent function $g$ in $B_{0}(\tau)$.  We claim  that $g$ is the restriction of some rational map  to $B_{0}(\tau)$ whose degree is not more than $d$.  Let us prove the claim. To this end, let us write
$$
f_{i}(z) = c_{i} z \frac{\prod_{k} (z - p^{i}_{k})} {\prod_{j}(z - q^{i}_{j})}
$$
where $c_{i} \ne 0$ and all the $p_{k}^{i}$ and $q_{j}^{i}$ do not belong to $B_{0}(\tau)$.  By taking a subsequence we may assume that as $ i \to \infty$, each of the $p_{k}^{i}$ and $q_{j}^{i}$ either converges to the infinity or converges to some complex number in the outside of $B_{0}(\tau)$. Let us denote this as $p_{k'}^{i} \to \infty$, $q_{j'}^{i} \to \infty$, $p_{k''}^{i} \to p_{k''}$, and $q_{j''}^{i} \to q_{j''}$ where the $p_{k''}$ and $q_{j''}$ are complex numbers in the outside of $B_{0}(\tau)$.
Since when restricted to $B_{0}(\tau)$ $f_{i}$ converges to $g$ and $\frac{\prod_{k''} (z - p_{k''}^{i})}{\prod_{j''}(z - q_{j''}^{i})}$ converges to $ \frac{\prod_{k''} (z - p_{k''})} {\prod_{j''}(z - q_{j''})}$, it follows that as $i \to \infty$,
$$
c_{i} \cdot \frac{\prod_{k'}p_{k'}^{i}}{\prod_{j'}q_{j'}^{i}} \to \alpha
$$
where $\alpha$ is some nonzero complex number. This implies that in $B_{0}(\tau)$, the univalent function $g$ is identified with the following rational function whose degree is clearly not more than $d$,
$$
\alpha z \frac{\prod_{k''} (z - p_{k''})} {\prod_{j''}(z - q_{j''})}.
$$
The claim has been proved. In the following let us still use $g$ to denote this rational function.

By taking a convergent subsequence if necessary, we may
assume that $\phi_{i}\to \phi$  uniformly  in any compact subset of
the unit disk where $\phi$ is some univalent  function defined in the unit
disk.  In particular, in a small neighborhood of the origin, $g(z) =
(\phi\circ R_{\theta}\circ \phi^{-1})(z)$ where $R_{\theta}$ is the
rigid rotation given by $\theta$.  Since $g$ is a rational map, it follows that
\begin{equation}\label{lse}
g(z) =
(\phi\circ R_{\theta}\circ \phi^{-1})(z) \hbox{ for all } z \in \phi(\Delta).
\end{equation} Since  $\phi_{i}\to \phi$  uniformly  in any compact subset of
the unit disk, $f_{i}$ converges uniformly to $g$ in any compact subset of $\phi(\Delta)$. There are three cases.

In the first case,  $g$ is a M\"{o}bius map. Since $g(0) = 0$ and $g'(0) = e^{2 \pi i \theta}$, it follows that $g$ has two distinct fixed points $\{0, p\}$, and moreover, $\widehat{\Bbb C} - \{0, p\}$ is foliated by $g$-invariant Euclidean circles.  Since $\phi_{i} \to
\phi$ uniformly in any compact subset of the unit disk,  it follows
 that $\Gamma_{r}^{i}$ converges to a Euclidean circle $\Gamma$
which is preserved by $g$, and moreover, $f_{i}$ uniformly converges
to $g$ in an open neighborhood of $\Gamma$. Since $g$ is conjugate
to the rigid rotation $R_{\theta}$ through a M\"{o}bius map, we thus
have
$$
\lim_{i \to \infty} V_{f_{i}}(r;k,l,m,n) =
V_{R_{\theta}}(r;k,l,m,n).
$$
 The Lemma in this case then follows from
Lemma~\ref{p} and the fact that the Euclidean circle is a
quasi-circle.

In the second case,  $g \in R_{\theta}^{d'}$ for some $2 \le d' <
d$.  Let $D^{g}$ denote the Siegel disk of $g$ centered at the origin.  By (\ref{lse}), it follows that $D^{g}$ always contains $\phi(\Delta)$ and may be strictly larger than $\phi(\Delta)$. Again since $\phi_{i} \to \phi$ uniformly in any compact subset
of the unit disk, it follows that  $\Gamma_{r}^{i}$ converges to the
$\Gamma_{r'}$ of $D^{g}$ for
some $0 < r' \le r$,
 and moreover, $f_{i}$ uniformly converges
to $g$ in an open neighborhood of $\Gamma_{r'}$. This implies that
\begin{equation}\label{pp-ii}
\lim_{i \to \infty} V_{f_{i}}(r;k,l,m,n) =  V_{g}(r';k,l,m,n) \ge
V_{g}(r;k,l,m,n).
\end{equation}
Since $g$ is a rational map with degree less than $d$, by induction
on the degree of the rational map we have a constant $0< C< \infty$
depending only on $d$ and $\theta$ such that $$V_{g}(r;k,l,m,n) >
C.$$ Thus the Lemma also follows in this case.

In the third case, $g \in R_{\theta}^{d}$. Then we still have
(\ref{pp-ii}). Thus  we get
\begin{equation}\label{cont-rr}
V_{g}(r;k,l,m,n) = \inf_{f \in R_{\theta}^{d}}\{V_{f}(r;k,l,m,n)\}.
\end{equation}
Recall that  $D^{g}$ denotes the Siegel disk of $g$ centered at the origin. By a M\"{o}bius conjugation which preserves $0$, we may assume that $\infty \notin D^{g}$ and $g(\infty) \ne \infty$. Let
 $\Gamma_{r}^{g}$  and $D_{r}^{g}$  denote the $\Gamma_{r}$ and
the $D_{r}$ of $D^{g}$ respectively. If $P_{g}\cap D_{r}^{g} =
\emptyset$, then $\Gamma_{r}^{g}$ is a $K(d, \theta)$-quasi-circle
by Lemma~\ref{con-ext}. The Lemma in this case then follows from
Lemma~\ref{p}.  Now suppose
$$
P_{g} \cap D_{r}^{g} \ne \emptyset.
$$
Let $V_{1}, \cdots, V_{N}$ denote all the components of
$g^{-1}(D^{g}_{r})$ in the outside of $D^{g}$ such that
$$
V_{i} \cap (\Omega_{g} \cup P_{g}) \ne
\emptyset, \quad i=1, \cdots, N.
$$
For each $1 \le i \le N$,  let
$$
g(V_{i} \cap (\Omega_{g} \cup P_{g})) = \{x_{1}, \cdots, x_{k_{i}}\}
$$
where $k_{i} \ge 1$ is some integer. For each $i$, take $k_{i}$
distinct  points in $\Gamma_{r}^{g}$, say $z_{1}^{i}, \cdots,
z_{k_{i}}^{i}$.

Now take an $r'$ such that $r< r' < 1$. For each $1 \le i \le N$,  take $k_{i}$ disjoint Jordan domains with smooth
boundaries, say $U_{1}^{i}, \cdots, U_{k_{i}}^{i}$ such that
$\overline{U_{j}^{i}} \subset D^{g}_{r'}$ and $\{x_{j}^{i},
z_{j}^{i}\} \subset U_{j}^{i}$ for all $1 \le j \le k_{i}$,
 and most
importantly,
$$
d_{U_{j}^{i}}(x_{j}^{i}, z_{j}^{i}) \equiv C_{0}
$$
holds for all $1 \le i \le N$ and  $1 \le j \le k_{i}$, where $d_{U_{j}^{i}}(\cdot, \cdot)$ denotes the distance with respect to the hyperbolic metric in $U_{j}^{i}$. In fact,
when the domain becomes thinner, the hyperbolic distance between the
two points will become bigger. So it is easy to make all
$d_{U_{l}^{i}}(x_{j}^{i}, z_{j}^{i})$ taking the same large value by
making all the domains $U_{j}^{i}$ thin enough. It follows that
there is a $t_{0} \in \Delta$ such that for each $U_{j}^{i}$, there
is a Riemann isomorphism
$$
\psi_{j}^{i}: \Delta \to U_{j}^{i}
$$
 such that $\psi_{j}^{i}(0) = x_{j}^{i}$ and
$\psi_{j}^{i}(t_{0}) = z_{j}^{i}$. Let $\phi_{j}^{i}$ denote the
inverse of $\psi_{j}^{i}$.  For each $1 \le i \le N$, define
$$
\Phi_{i}(\cdot, \cdot): D^{g}_{r'} \times \Delta  \to D^{g}_{r'}
$$ as follows
$$
\Phi_{i}(z, t) =
\begin{cases}
 z & \text{ if $z \in D^{g}_{r'} \setminus  \bigcup_{1 \le j \le k_{i}} U_{j}^{i}$}, \\
 \psi_{j}^{i}\big{(}\phi_{j}^{i}(z) + (1-|\phi_{j}^{i}(z)|)t\big{)} & \text{ if $z \in U_{j}^{i}$  for some $1 \le j \le k_{i}$}.
\end{cases}
$$
By a direct calculation, we have
\begin{equation}\label{ppdd}
\frac{ (\Phi_{i})_{\bar{z}}}{ (\Phi_{i})_{z} }(z, t) =
\begin{cases}
 0 & \text{ if $z \in D^{g}_{r'} \setminus  \bigcup_{1 \le j \le k_{i}} U_{j}^{i}$}, \\

  \frac{\overline{(\phi_{j}^{i})'(z)}}{(\phi_{j}^{i})'(z)} \frac{t \phi_{j}^{i}(z)}{t\overline{{\phi_{j}^{i}(z)}} - 2|\phi_{j}^{i}(z)|} & \text{ if $z \in U_{j}^{i}$  for some $1 \le j \le k_{i}$}.
\end{cases}
\end{equation}
This implies that for almost every $z$ in $D^{g}_{r'}$, the complex dilatation of $\Phi_{i}$ depends analytically on $t$ when $t$ varies in $\Delta$.  For each $1 \le i \le N$, let $V_{i}'$ be the component of $g^{-1}(D^{g}_{r'})$ which contains $V_{i}$. Since $V_{i}$ is in the outside of $D^{g}$, we have $V_{i}' \cap D^{g} = \emptyset$ for all $1 \le i \le N$.  For each $t \in \Delta$, define
\begin{equation}\label{jj-k}
h_{t}(z) =
\begin{cases}
g(z) & \text{ if $z \in \widehat{\Bbb C} \setminus  \bigcup_{1 \le i \le N} V_{i}'$}, \\
\Phi_{i}(g(z), t) & \text{ if $z \in V_{i}'$ for some $1 \le i \le N$}.
\end{cases}
\end{equation}
It follows that $h_{t}:  \widehat{\Bbb C} \to \widehat{\Bbb C}$ is a branched covering map of degree $d$.
Let
$$
\Omega = \bigcup_{1 \le i \le N} V_{i}' \cap \big{(} \bigcup _{1 \le j \le k_{i}} g^{-1}(U_{j}^{i})\big{)}
$$
Since all the $\partial U_{j}^{i}$ are smooth Jordan curves, $\partial \Omega$ is the union of finitely many quasi-circles. For each $t \in \Delta$, from (\ref{ppdd}) and (\ref{jj-k}) we can easily get
$$
\frac{ (h_{t})_{\bar{z}} } {(h_{t})_{{z}} }(z) =
\begin{cases}
 0 & \text{ if $z \in \widehat{\Bbb C} \setminus  \Omega$}, \\

  \frac{\overline{(\phi_{j}^{i})'(g(z))}}{(\phi_{j}^{i})'(g(z))} \frac{t\phi_{j}^{i}(g(z))}{t\overline{\phi_{j}^{i}(g(z))} - 2|\phi_{j}^{i}(g(z))|}\frac{\overline{g'(z)}}{g'(z)} & \text{ if $z \in \Omega$}.
\end{cases}
$$
This implies that for every $t \in \Delta$, the map $h_{t}: \widehat{\Bbb C} \to \widehat{\Bbb C}$ is a quasi-regular branched covering map of degree $d$,  and moreover, for almost every $z$, the complex dilatation of $h_{t}$ at $z$ depends analytically on $t$ when $t$ varies in $\Delta$.

By the construction of $h_{t}$, it follows that for each $t \in \Delta$, $h_{t}|D^{g} = g|D^{g}$ is conformal in $D^{g}$,  and moreover, for almost every $z \in \widehat{\Bbb C}$, if $h_{t}$ is quasiconformal at some point $z$, then $h_{t}(z) \in D^{g}$.
So for each $t \in \Delta$,  by pulling  back the standard complex
structure $\mu_{0}$ in the Siegel disk $D^{g}$  through the iteration of
$h_{t}$, we can get a $h_{t}$-invariant complex structure $\mu_{t}$
in the whole sphere. Again by a direct calculation we get
$$
\mu_{t}(z) =
\begin{cases}
\frac{\overline{(g^{n})'(z)}}{(g^{n})'(z)}  \frac{ ( h_{t})_{\bar{z}}}{ (h_{t})_{z}}(g^{n}(z))   & \text{ if $g^{n}(z) \in \Omega$ for some integer $n \ge 0$}, \\
0 & \text{ if otherwise}.
\end{cases}
$$
From the above formula, it follows that for almost every  $z\in \widehat{\Bbb C}$,
$\mu_{t}(z)$ depends analytically on $t$. Let $\phi_{t}$ be the
quasiconformal homeomorphism of the sphere which fixes $0$, $1$ and
the infinity and which solves the Beltrami equation given by
$\mu_{t}$. Then $\phi_{t}$ depends analytically on $t$. Let
$$
g_{t}(z) = \phi_{t} \circ h_{t} \circ \phi_{t}^{-1}(z).
$$
We claim
\begin{itemize}
\item[1.] $g_{0} = g$,
\item[2.] $g_{t} \in R_{\theta}^{d}$ for each $t \in \Delta$,
\item[3.] $g_{t}$ depends analytically on $t$ when $t$ varies in $\Delta$,
\item[4.] The post-critical set of $g_{t_{0}}$  does not intersect the $D_{r}$ of $g_{t_{0}}$.
\end{itemize}

Let us prove the claim. The first two assertions follow directly from the construction. Let us prove the third assertion (We would like to remark here that $\phi_{t}$ depends analytically on $t$ does not imply that $\phi_{t}^{-1}$ depends analytically on $t$ also). Note that $\infty \notin D^{g}$ and $g(\infty) \ne \infty$ by the assumption right behind (\ref{cont-rr}).  Take $p \in \Bbb C$ such that $p \notin D^{g}$ and $p \ne g(\infty)$. Let $a_{1}, \cdots, a_{d}$, counted by multiplicities,  be all the $p$-value points of $g$, that is, $g(a_{i}) = p$ for $1 \le i \le d$. Let $b_{i}, 1 \le i \le d$, be all the poles of $g$, again counted by multiplicities.  Since $\infty \notin D^{g}$ and $g(\infty) \ne \infty$, all the $a_{i}$ and $b_{i}$ are complex numbers.    Since both $p$ and $\infty$ do not belong to $D^{g}$,  by the definition of $h_{t}$, it follows that $h_{t}(a_{i}) = g(a_{i}) = p$ and $h_{t}(b_{i}) = g(b_{i}) = \infty$ for all $t \in \Delta$ and $1 \le i \le d$.   Then $\phi_{t}(a_{i}), 1 \le i \le d$, are all the $\phi_{t}(p)$-value points of $g_{t}$, and $\phi_{t}(b_{i}), 1 \le i \le d$, are all the poles of $g_{t}$. Since all the $a_{i}$ and $b_{i}$ do not belong to $D^{g}$, it follows  that all the $\phi_{t}(a_{i})$ and $\phi_{t}(b_{i})$ do not belong to the Siegel disk of $g_{t}$ centered at the origin, and thus are all non-zero complex numbers. Let $$c(t) =  \prod_{i=1}^{d}\frac{\phi_{t}(a_{i})}{\phi_{t}(b_{i})}.$$ Since $g_{t}(0) = 0$, it follows that
$$
g_{t}(z) = \phi_{t}(p) - \frac{\phi_{t}(p)}{c(t)} \cdot \prod_{i=1}^{d} \frac{(z - \phi_{t}(a_{i}))}{(z - \phi_{t}(b_{i}))}.
$$
This implies that $g_{t}$ depends analytically on $t$. The third assertion follows. Now let us prove the last assertion. First note that  $h_{t_{0}}| D^{g} = g|D^{g}$,  and
$$
(\Omega_{h_{t_{0}}} \cup P_{h_{t_{0}}}) - D^{g} = (\Omega_{g} \cup P_{g}) - D^{g}.
$$
Suppose $z \in  (\Omega_{h_{t_{0}}} \cup P_{h_{t_{0}}}) - D^{g}$ is a point such that $g(z) \in D^{g}$. By the previous construction it follows that $z \in V_{i}$ for some $1 \le i \le N$ and  $g(z) = x_{j}^{i}$ for some $1 \le j \le k_{i}$. So $h_{t_{0}}(z) = \Phi_{i}(g(z), t_{0}) = z_{j}^{i}$ belongs to the $\Gamma_{r}$ of $D^{g}$.  Note that $\phi_{t_{0}}: D^{g}\to D^{g_{t_{0}}}$ is a holomorphic isomorphism and is the conjugation map between  $g|D^{g} = h_{t_{0}}|D^{g}: D^{g} \to D^{g}$  and $g_{t_{0}}|D^{g_{t_{0}}}: D^{g_{t_{0}}} \to D^{g_{t_{0}}}$. So $\phi_{t_{0}}$ maps the $\Gamma_{r}$ of $D^{g}$ to the $\Gamma_{r}$ of $D^{g_{t_{0}}}$. In particular, $g_{t_{0}}(\phi_{t_{0}}(z)) = \phi_{t_{0}}(z_{j}^{i})$ belongs to the $\Gamma_{r}$ of $D^{g_{t_{0}}}$. The last assertion of the claim has been proved. The proof of the claim is completed.

Now take $z_{0}$ in the $\Gamma_{r}$ of $D_{g}$ such that
$$|C_{g;k,l,m,n}(z_{0})| = V_{g}(r;k,l,m,n).$$ Since for any given $z$, $\phi_{t}(z)$ is holomorphic in
$t$ for $t \in \Delta$,   it follows that for every integer $i \ge
0$, the map $g_{t}^{i}(\phi_{t}(z_{0})) = \phi_{t}(g^{i}(z_{0}))$ is
holomorphic in $t$ for $t \in \Delta$. Thus the map
$$
C_{g_{t};k,l,m,n}(\phi_{t}(z_{0})) =  \frac{(g_{t}^{k}(\phi_{t}(z_{0}))-g_{t}^{m}(\phi_{t}(z_{0})))
(g_{t}^{l}(\phi_{t}(z_{0}))-g_{t}^{n}(\phi_{t}(z_{0})))}
{(g_{t}^{k}(\phi_{t}(z_{0}))-g_{t}^{n}(\phi_{t}(z_{0})))(g_{t}^{l}(\phi_{t}(z_{0}))-g_{t}^{m}(\phi_{t}(z_{0})))}
$$
is a  holomorphic function in $t$ which does not vanish for $t \in \Delta$.  Since
$\phi_{t}$ maps the $\Gamma_{r}$ of $D^{g}$ to the $\Gamma_{r}$ of
$D^{g_{t}}$,   $\phi_{t}(z_{0})$ belong to the $\Gamma_{r}$ of $D^{g_{t}}$. We thus have
$$
|C_{g_{t};k,l,m,n}(\phi_{t}(z_{0}))| \ge V_{g_{t}}(r;k,l,m,n) \hbox{  for all } t \in \Delta.
$$

This, together with  (\ref{cont-rr}) and the choice of $z_{0}$,
implies that the modulus of the  holomorphic function
$C_{g_{t};k,l,m,n}(\phi_{t}(z_{0}))$ obtains the minimum at $t = 0$.
Since $C_{g_{t};k,l,m,n}(\phi_{t}(z_{0}))$ does not vanish for $t \in \Delta$, it follows  that $C_{g_{t};k,l,m,n}(\phi_{t}(z_{0}))$
is a constant function.  In particular, we have
$$
|C_{g;k,l,m,n}(z_{0})|  =
|C_{g_{_{t_{0}}};k,l,m,n}(\phi_{t_{0}}(z_{0}))| \ge
V_{g_{_{t_{0}}}}(r;k,l,m,n).
$$
But by the last assertion of the claim we just proved, the postcritical set of $g_{t_{0}}$ does
not intersect the $D_{r}$ of $D^{g_{t_{0}}}$. By Lemma~\ref{con-ext}
there is  a $1 < C < \infty$ depending only on $d$ and $\theta$ such
that
$$
V_{g_{_{t_{0}}}}(r;k,l,m,n) > C.
$$
This proves the lemma in the third case. The proof of
Lemma~\ref{post-in} is completed.
\end{proof}

Now let us prove the Main Theorem. Since the forward orbit of any
$z$ in $\Gamma_{r}$ is dense in $\Gamma_{r}$, it follows from
Lemma~\ref{post-in} and Lemma~\ref{p} that there is a $1< K(d,
\theta)<\infty$ depending only on $d$ and $\theta$ such that every
$\Gamma_{r}$ is a $K(d, \theta)$-quasi-circle. The Main Theorem then
follow from Lemma~\ref{inv}.
%%%%%%%%%%%%%%%%%%%%%%%%%%%%%%%%%%%%%%%%%%%%%%%%%%%%%%%%%%%%%%%%%%%%%%
%%%%%%%%%%%%%%%%%%%%%%%%%%%%%%%%%%%%%%%%%%%%%%%%%%%%%%%%%%%%%%%%%%%%%%
%%%%%%%%%%%%%%%%%%%%%%%%%%%%%%%%%%%%%%%%%%%%%%%%%%%%%%%%%%%%%%%%%%%%%%
%%%%%%%%%%%%%%%%%%%%%%%%%%%%%%%%%%%%%%%%%%%%%%%%%%%%%%%%%%%%%%%%%%%%%%

\section{Proof of Theorem B}
\subsection{From Cross Ratios to Simple closed Geodesics}  For two distinct points $a, b \in \Bbb T$, let $[a, b]$
denote the arc segment which connects $a$ and $b$ in anti-clockwise
direction. For an arc
segment $I \subset {\Bbb T}$, let  $|I|$  denote the length of
$I$ with respect to the Euclidean metric. We say an arc segment $J$ is properly contained in $I$ if $J \subset I$ and $I \setminus J$ consists of two non-trivial arc segments.  In this case, we denote it by $J \Subset I$.

Now for any two arc segments $J \Subset I \subset
\Bbb T$, we define
$$
C(I, J) = \frac{|I||J|}{|R||L|}
$$
where  $R$ and $L$ denote the two arc components of $I - J$, respectively.
From the definition, we have
\begin{lemma}\label{space}
Let $0< K< \infty$. Then for any arc segments $J \Subset I \subset
\Bbb T$, if $C(I, J) < K$, we have
$$
\min \{|R|, |L|\} >   |J|/K.
$$
\end{lemma}
By the above lemma,  it follows that the value $C(I, J)$ measures
the space around $J$ in $I$.

Let $B \in {\mathbf{B}}_{\theta}^{m}$.  Let $k \ge 1$ be an integer
and $S, T \subset \Bbb T$ be two arc segments. We say $S$ is the
pull back of $T$ by $B^{k}$ if $B^{k}: S \to T$ is a homeomorphism.
Suppose  $J \Subset I \subset \Bbb T$ are two arc segments. Let
us denote them by $I^{0}_{B}$ and $J^{0}_{B}$ respectively. For $k \ge 1$, let $I^{k}_{B}$ and $J^{k}_{B}$ denote the arc segments in $\Bbb T$
which are the pull backs of $I^{0}_{B}$ and $J^{0}_{B}$ respectively by $B^{k}$. The next lemma is the key in
the proof of Theorem B.

\begin{lemma}\label{Swiatek-D}
Let $m = 2d -1 \ge 3$ be an odd integer and $0< \theta <1$ be a bounded type
irrational number.  Then, there exist constants $\alpha \in (0, \infty)$ and $\beta \in (0, \infty)$ depending only on $m$ and $\theta$, such that for any centered Blaschke product $B \in {\mathbf{B}}_{\theta}^{m}$ and any disjoint family of arc segments $\{I_{B}^{k}\:|\:0\le k \le N\}$ and any family of arc segments   $\{J_{B}^{k}\:|\:0\le k \le N\}$  with $J_{B}^{k} \Subset I_{B}^{k}$ for all $0 \le k \le N$,  we have
$$
C(I_{B}^{N}, J_{B}^{N}) \le \beta \cdot (1 + C(I_{B}^{0}, J_{B}^{0})^{\alpha}.
$$
\end{lemma}

The main task in the proof of Theorem  A in \cite{H} is to prove
that the \'{S}wiatek distortion has a uniform upper bound for all the
Blaschke products in ${\mathbf{H}}_{\theta}^{m}$.  The difference
between the two situations is that Herman's proof uses real
techniques and relies essentially on the compact property of
${\mathbf{H}}_{\theta}^{m}$, which does not hold for
${\mathbf{B}}_{\theta}^{m}$ (see
 Lemma~\ref{comp-H} and Remark~\ref{rdd}). To solve this problem, we
make  use of the complex analytic property of the maps in
${\mathbf{B}}_{\theta}^{m}$. Instead of considering  cross ratios,
we consider the length of certain simple closed geodesics. As a
result, we reduce Lemma~\ref{Swiatek-D} to showing that  the length
of certain simple closed geodesics, after disjoint pull backs, can
be increased by at most some factor which is bounded above by a
constant depending only on  $m$ (Lemma~\ref{lemma:ex}). Let us
introduce some notations before we expose this idea further.

Let $\widehat{\Bbb C}$ denote the Riemann sphere. Let $B \in
{\mathbf{B}}_{\theta}^{m}$ be a centered Blaschke product.  For $0 \le k \le N$, let $I^{k}_{B}$ and
$J^{k}_{B}$ be the arc segments given in
Lemma~\ref{Swiatek-D}. Let
$$
X^{k}_{B} =(\widehat{\Bbb C} - \Bbb T) \cup (I_{B}^{k} - J_{B}^{k})
$$
Then there exists a unique simple closed geodesic in $X^{k}_{B}$
which separates $J^{k}_{B}$ and ${\Bbb T} - I^{k}_{B}$.  Let us
denote it by $\gamma^{k}_{B}$. Let $l_{X^{k}_{B}}(\gamma^{k}_{B})$
denote the length of $\gamma^{k}_{B}$ with respect to the hyperbolic
metric in $X^{k}_{B}$. The goal of this section is to reduce
Lemma~\ref{Swiatek-D} to the following lemma.

\begin{lemma}\label{lemma:ex}
Let $m=2d-1 \ge 3$ be an odd integer and $0< \theta <1$ be a bounded
type irrational number. Then there exists a $1< C(m) < \infty$ which
depends only on $m$ such that for any  Blaschke product $B \in
{\mathbf{B}}_{\theta}^{m}$, and any disjoint family of arc segments
$\{I^{k}_{B}\:\big{|}\: 0\le k \le N\}$ and any family of arc segments $\{J_{B}^{k}\:|\:0\le k \le N\}$
with $J_{B}^{k} \Subset I_{B}^{k}$ for all $0 \le k \le N$, we have
$$
\frac{l_{X^{N}_{B}}(\gamma^{N}_{B})} {l_{X^{0}_{B}}(\gamma^{0}_{B})}
\le C(m).
$$

\end{lemma}

\begin{proposition}\label{bridge}{\rm
Lemma~\ref{lemma:ex} implies Lemma~\ref{Swiatek-D}.}
\end{proposition}
We need to prove Lemmas \ref{Tech}-\ref{centered} before we prove
Proposition~\ref{bridge}.
For $T \in (0, \infty)$, let $\Lambda(T)$ be the modulus of the annulus
$\Bbb C \setminus ([-1, 0] \cup [T, \infty))$.

\begin{lemma}\label{Tech} For all $T \in (0, \infty)$, we have
$$\Lambda(T) \cdot \Lambda(1/T) = 1/4 \:\hbox{  {\rm and}   }\:  T <  e^{2 \pi \Lambda(T)} \le 16 (T+1).$$
\end{lemma}
\begin{proof}
See Chapter III of \cite{A1}.
\end{proof}

\begin{lemma}\label{length-modulus}
Let $A \subset \widehat{\Bbb C}$ be an annulus and $\gamma \subset A$ be its core geodesic. Then
$$
l_{A}(\gamma) = \frac{\pi}{{\rm mod}(A)}
$$
where $l_{A}(\gamma)$ is the length of $\gamma$ with respect to the hyperbolic metric in $A$.
\end{lemma}
\begin{proof}
We may assume that $A$ is a Euclidean annulus $\{z\:|\: e^{-\alpha} < |z| < e^{\alpha}\}$ for some $\alpha > 0$. It follows that $${\rm mod}(A) = \frac{1}{2 \pi} \log \frac{e^{\alpha}}{e^{-\alpha}} = \frac{\alpha}{\pi}.$$  To compute the length of the core geodesic $\gamma$ of $A$, consider the vertical strip
$$
S = \{z = x + iy \: |\: -\alpha < x < \alpha, -\infty < y< +\infty\}.
$$
The map $\Phi: z \mapsto e^{z}$ is a holomorphic covering map from $S$ to $A$.  Let  $\Gamma = [-\pi i, \pi i]$  be the vertical straight segment.   It is clear that $l_{S}(\Gamma) = l_{A}(\gamma)$. To compute $l_{S}(\Gamma)$, let us consider the map $$\Psi: w \mapsto e^{i \frac{\pi}{2\alpha} w}.$$ The map $\Psi$ maps $S$ isomorphically to the right half plane $H$. Under this map, the vertical straight segment $\Gamma$ is mapped to the horizontal straight segment $\Gamma' = [e^{-\frac{\pi^{2}}{2 \alpha}}, e^{\frac{\pi^{2}}{2 \alpha}}]$. We thus have
$$
l_{A}(\gamma) = l_{S}(\Gamma) = l_{H}(\Gamma') = \int_{e^{-\frac{\pi^{2}}{2 \alpha}}}^{e^{\frac{\pi^{2}}{2 \alpha}}} \frac{1}{x} dx  = \frac{\pi^{2}}{\alpha} = \frac{\pi}{{\rm  mod}(A)}.
$$
This completes the proof of Lemma~\ref{length-modulus}.
\end{proof}

\begin{lemma}\label{re-s}
For any arc segments $J \Subset I \subset \Bbb T$, we have
$$
\frac{(2 \pi - |I|)^{2}}{4 \pi^{2}} \cdot C(I, J) \le e^{l_{X}(\gamma)/2} \le 4 \pi^{2} \cdot (1 + C(I, J)),
$$
where $X = (\widehat{\Bbb C} - \Bbb T) \cup (I - J)$ and $l_{X}(\cdot)$ denotes the length with respect to the hyperbolic in $X$.
\end{lemma}
\begin{proof}
Assume that $I = [e^{i\theta_{1}}, e^{i \theta_{4}}]$ and $J = [e^{i \theta_{2}}, e^{i \theta_{3}}]$ and assume that $0\le \theta_{1} < \theta_{2}< \theta_{3} < \theta_{4} \le 2 \pi$.   Let $M$ be the M\"{o}bius transformation sending $e^{i \theta_{2}}$ to $0$, $e^{i \theta_{3}}$ to $-1$, and $e^{i \theta_{4}}$ to $\infty$. Then $M(e^{i \theta_{1}}) \in (0, +\infty)$. Let $T = 1/M(e^{i \theta_{1}})$. By Lemmas~\ref{Tech} and \ref{length-modulus} it follows that
$$
l_{X}(\gamma) = \frac{\pi}{\Lambda(1/T)} = 4 \pi \Lambda (T).
$$
This, together with the second inequality of Lemma~\ref{Tech}, implies
$$
T < e^{2 \pi \Lambda (T)} = e^{l_{X}(\gamma)/2} \le 16(T+1).
$$
Since the cross ratio is preserved by M\"{o}bius transformation, it follows that $$
T = \bigg{|}\frac{(e^{i \theta_{3}} - e^{i \theta_{2}})(e^{i \theta_{4}} - e^{i \theta_{1}})}{(e^{i \theta_{4}} - e^{i \theta_{3}})(e^{i \theta_{1}} - e^{i \theta_{2}})} \bigg{|}
$$
Since $|I| = \theta_{4} - \theta_{1}$, $|J| = \theta_{3} - \theta_{2}$, $|R| = \theta_{4} - \theta_{3}$ and $|L| = \theta_{2} - \theta_{1}$, we have
\begin{equation}\label{skr}
T = \bigg{|}\frac{(e^{i |I|} - 1)(e^{i |J|} - 1)}{(e^{i |R|} - 1)(e^{i |L|} - 1)} \bigg{|} = \bigg{|}\frac{\sin(|I|/2) \sin |J|/2)}{\sin|R|/2) \sin(|L|/2)} \bigg{|}.
\end{equation}

Note that for $x \in (0, 2 \pi)$, we have $4 \pi \sin(x/2) \ge x (2 \pi - x)$ and $0 \le \sin(x/2) \le x/2$. Both the inequalities can be easily proved by calculus and we shall leave the proofs to the reader.  From these two inequalities and (\ref{skr}) we get
$$
T \ge \frac{1}{4 \pi^{2}} \frac{|I|(2\pi - |I|) \cdot |J|(2 \pi - |J|)}{|L|\cdot |R|} \ge \frac{1}{4 \pi^{2}} \cdot C(I, J) \cdot (2 \pi - |I|)^{2}.
$$
Since $T < e^{l_{X}(\gamma)/2}$, it follows that
\begin{equation}\label{lt-in}
\frac{(2 \pi - |I|)^{2}}{4 \pi^{2}} \cdot C(I, J) \le e^{ l_{X}(\gamma)/2}.
\end{equation}

Note that for $x \in [0, \pi]$, we have $x/\pi \le \sin(x/2) \le x/2$. Again the inequality can be easily proved by calculus and we omit the proof here. Thus, if $|L| \le \pi$ and $|R| \le \pi$, from this inequality and (\ref{skr}) we get
$$
T \le \frac{\pi^{2}}{4} C(I, J).
$$
If $\pi \le |L| \le |I|$, then $|R| \le \pi$ and
$$
T \le \frac{\sin(|J|/2)}{\sin(|R|/2)} \le \frac{\pi}{2}\frac{|J|}{|R|} \le \frac{\pi}{2}C(I,J).
$$
If $\pi \le |R| \le |I|$, then $|L| \le \pi$ and
$$
T \le \frac{\sin(|J|/2)}{\sin(|L|/2)} \le \frac{\pi}{2}\frac{|J|}{|L|} \le \frac{\pi}{2}C(I,J).
$$

In all the cases we have
\begin{equation}\label{rt-in}
e^{l_{X}(\gamma)/2} \le 16(T+1) \le 4 \pi^{2}C(I,J) +16 < 4 \pi^{2}\cdot(1+C(I,J)).
\end{equation}
Lemma~\ref{re-s} then follows from (\ref{lt-in}) and (\ref{rt-in}).
\end{proof}

For any $B \in {\mathbf{B}}_{\theta}^{m}$, recall that $\mu_{B}$ is the invariant probability measure of $B|\Bbb T: \Bbb T \to \Bbb T$.
\begin{lemma}\label{centered}
Assume that $B \in {\mathbf{B}}_{\theta}^{m}$ is centered and $I \subset \Bbb T$ is an arc segment such that  $\mu_{B}(I) < \delta \le 1/2$. Then
$$
|\Bbb T - I| \ge 2 \arccos \frac{\delta}{1 - \delta}.
$$
\end{lemma}
\begin{proof}
Set $\eta = \mu_{B}(I)$. Then $\eta \le \delta$ and $\mu_{B}(\Bbb T - I) = 1 - \eta$. Set $L = |\Bbb T - I|$ and without loss of generality, let us assume that $\Bbb T - I = [e^{-L/2}, e^{L/2}]$ is the arc segment in $\Bbb T$ which connects $e^{-L/2}$ and $e^{L/2}$ anticlockwise.  Since $0 \le L/2 \le \pi$ and the function $x \mapsto \cos(x)$ is decreasing  on $[0, \pi]$, it follows that for every $z \in \Bbb T - I$, one has
$$
\Re (z) \ge \cos (L/2).
$$ It is clear that $\Re(z) \ge -1$ for all $z \in I$. Since $B$ is centered, by (\ref{clt}) we have  $ \int_{\Bbb T} z d\mu_{B}(z) = 0$. We thus get
$$
 \int_{\Bbb T} \Re(z) d\mu_{B}(z)  =  \Re\big{(} \int_{\Bbb T} z d\mu_{B}(z)\big{)} = 0.
$$ Since
$$
 \int_{\Bbb T} \Re(z) d\mu_{B}(z) =  \int_{\Bbb T - I} \Re(z) d\mu_{B}(z) + \int_{I} \Re(z) d\mu_{B}(z) \ge (1 - \eta) \cos(L/2) - \eta,
$$ we have $$(1 - \eta) \cos(L/2) - \eta \le 0. $$ This implies that
$$
\cos(L/2) \le \frac{\eta}{1 - \eta} \le \frac{\delta}{1 - \delta}
$$
and thus
$$
L \ge 2 \arccos \frac{\delta}{1 - \delta}.
$$
Lemma~\ref{centered} follows.
\end{proof}

Now it is the time to prove Proposition~\ref{bridge}.
\begin{proof}
If $N = 0$, the result is trivial. So let us assume that $N \ge 1$. Since $I_{B}^{N}$ is disjoint from $I_{B}^{N-1}$, we have that
$$
\mu_{B}(I_{B}^{N}) \le \delta = \min\{\theta, 1 - \theta\} < 1/2.
$$
According to Lemma~\ref{centered}, we have
$$
2 \pi - |I_{B}^{N}| = |\Bbb T - I_{B}^{N}| \ge \epsilon = 2 \arccos \frac{\delta}{1 - \delta}.
$$
According to Lemma~\ref{re-s}, we have
$$
C(I_{B}^{N}, J_{B}^{N}) \le \frac{4 \pi^{2}}{\epsilon^{2}} e^{l_{X_{B}^{N}}(\gamma_{B}^{N})/2} \le \frac{4 \pi^{2}}{\epsilon^{2}} e^{\alpha \cdot l_{X_{B}^{0}}(\gamma_{B}^{0})/2}
$$
where $\alpha = C(m)$ is the constant provided by Lemma~\ref{lemma:ex}. The result then follows by taking  $\beta = \epsilon^{-2} \cdot (4 \pi^{2})^{1 + \alpha}$ since by Lemma~\ref{re-s}, we have
$$
e^{\alpha \cdot l_{X_{B}^{0}}(\gamma_{B}^{0})/2} \le (4 \pi^{2})^{\alpha}(1 + C(I_{B}^{0}, J_{B}^{0}))^{\alpha}.
$$
This completes the proof of Proposition~\ref{bridge}.
\end{proof}

%%%%%%%%%%%%%%%%%%%%%%%%%%%%%%%%%%%%%%%%%%%%%%%%%%%%%%%%%%%%%%%%%%%%
%%%%%%%%%%%%%%%%%%%%%%%%%%%%%%%%%%%%%%%%%%%%%%%%%%%%%%%%%%%%%%%%%%%%
%%%%%%%%%%%%%%%%%%%%%%%%%%%%%%%%%%%%%%%%%%%%%%%%%%%%%%%%%%%%%%%%%%%%
%%%%%%%%%%%%%%%%%%%%%%%%%%%%%%%%%%%%%%%%%%%%%%%%%%%%%%%%%%%%%%%%%%%%

\subsection{Proof of Lemma~\ref{lemma:ex}}

The proof of Lemma~\ref{lemma:ex} is based on Lemmas~\ref{lm-1}-\ref{nsd}. Before we state and prove these lemmas, let us introduce some common notations which will be used in all these lemmas.  Let $N \ge 1$ be an arbitrary integer. Let $J^{k} \Subset I^{k} \subset \Bbb T$, $0\le k \le N$, be arc segments such that all $I^{k}, 0 \le k \le N$,  are disjoint with each other.    Let $p \ge 1$ be an integer and $Z = \{z_{1}, \cdots, z_{p}\}$ be a finite subset of $\widehat{\Bbb C}$ containing $p$ points. For $0 \le k \le N$, we set
$$
U_{k} = (\widehat{\Bbb C} - \Bbb T) \cup (I^{k} - J^{k}) \:\:\hbox{ and }\:\: V_{k} = U_{k} - Z.
$$
We let $l_{k}$ be the length of the core geodesic of the annulus $U_{k}$ and $l_{k}'$ be the length of a shortest simple closed geodesic in $V_{k}$ separating $J^{k}$ and $\Bbb T - I^{k}$ (there may be several geodesics with minimal length). Note that $l_{k} \le l_{k}'$.

\begin{lemma}\label{lm-1}
Let $A$ be an annulus and $Z  = \{z_{1}, \cdots, z_{p}\} \subset A$. Then, there is an annulus $B \subset A \setminus Z$ homotopic to $A$ with
$$
{\rm mod}(A) \le (p+1) {\rm mod}(B).
$$
\end{lemma}
\begin{proof}
Without loss of generality, we may assume that $A$ is a round annulus $\{z\:|\: r < |z| < R\}$ for some $0 \le r < R$. Cutting $A$ along at most $p$ round circles passing through the points in $Z$, we find at most $p+1$ round annuli contained in $A - Z$, whose moduli add up to that of $A$. Let $B$ be one of those subannuli with maximal modulus. Then ${\rm mod}(A) \le (p+1) {\rm mod}(B)$. This completes the proof of Lemma~\ref{lm-1}.
\end{proof}

\begin{corollary}\label{cry-1}{\rm
For all $0\le k \le N$, we have $l_{k}' \le (p+1) \cdot l_{k}$.
}
\end{corollary}
\begin{proof}
Apply Lemma~\ref{lm-1} with $A = U_{k}$ and obtain an annulus $B \subset U_{k} - Z = V_{k}$ homotopic to $A$ such that ${\rm mod}(A) \le (p+1)  {\rm mod}(B)$. This implies that
$$
{\rm mod}(U_{k}) \le (p+1)  {\rm mod}(B).
$$
Let  $\gamma_{k}'$ be the core geodesic of $B$. Then by Lemma~\ref{length-modulus} we have
$$
l_{k}' \le l_{V_{k}}(\gamma') \le l_{B}(\gamma_{k}') = \frac{\pi}{{\rm mod}(B)} \le (p+1) \frac{\pi}{{\rm mod}(U_{k})} = (p+1) \cdot l_{k}.
$$
\end{proof}

\begin{definition}\label{key-domain}{\rm
Let $I \subset \Bbb T$ be an arc segment. Let $\Gamma$ be the unique Euclidean circle which passes through the end points of $I$ and is  orthogonal to the unit circle (In the case that $|I| = \pi$, $\Gamma$ is a straight line). We use $D(I)$ to denote the component of $\widehat{\Bbb C} - \Gamma$ which  contains the interior of $I$. }
\end{definition}
\begin{remark}\label{complement}{\rm
From the definition, it is clear that if $|I| < \pi$, $D(I)$ is a Euclidean disk;  if $|I| = \pi$, $D(I)$ is a half plane; and if $|I| > \pi$, $D(I)$ is the outside of a Euclidean disk. }
\end{remark}
\begin{lemma}\label{kgy}
Suppose that $J \Subset I$ are two  arc segments.  Let $\gamma$ be the core geodesic of $(\widehat{\Bbb C} - \Bbb T) \cup (I - J)$. Then $\gamma$ is a Euclidean circle orthogonal to the unit circle. In particular,  $\gamma \subset D(I)$.
\end{lemma}
\begin{proof}
Let $I = [a, d]$ and $J = [b, c]$. Let $\phi$ be a M\"{o}bius map which maps $a$ to $\infty$, $b$ to $-1$ and $c$ to $0$. Then $\phi$ maps $d$ to some point $T \in (0, +\infty)$ and maps the unit circle to the real line.  Let $\Gamma$ be the Euclidean circle with center $-1$ and radius $\sqrt{1 + T}$.  Note that $\Bbb C - ([-1, 0] \cup [T, \infty))$ is symmetric about $\Gamma$.

Let $\Omega$ be the disk $\{z\:|\:|z +1| < \sqrt{1+T}\}$. Let $H = \Omega \setminus [-1, 0]$. Let $0< r < 1$ be the number such that
$$
{\rm mod}(H) = \frac{1}{2 \pi} \log \frac{1}{r}.
$$
Let $\psi: H \to \{z\:|\: r < |z| < 1\}$ be the holomorphic isomorphism  such that
 the outer boundary component of $H$ is mapped to the unit circle. Then by Schwarz Reflection Lemma the map $\psi$ can be extended to a holomorphic isomorphism between $\Bbb C - ([-1, 0] \cup [T, \infty))$ and the annulus $\{z\:|\: r < |z| < r^{-1}\}$. In particular, $\psi$ maps $\Gamma$ to the unit circle which is the core geodesic of the annulus $\{z\:|\: r < |z| < r^{-1}\}$. This implies that $\Gamma$ is the core geodesic of $\Bbb C - ([-1, 0] \cup [T, \infty))$. This implies that $\phi^{-1}(\Gamma)$, which must be a Euclidean circle orthogonal to the unit circle, is the core geodesic of $(\widehat{\Bbb C} - \Bbb T) \cup (I - J)$. The proves the first assertion. The  second assertion follows directly  from the first assertion and the definition of  $D(I)$.   This completes the proof of Lemma~\ref{kgy}.
\end{proof}

For $z \in \widehat{\Bbb C} \setminus \Bbb T$, let $\phi_{z}$ be a M\"{o}bius map sending $z$ to $0$ and preserving $\Bbb T$. It is clear that $\phi_{z}$ is unique up to a post-composition with a rotation. For an arc segment $I \subset \Bbb T$,  set $$\mu_{z}(I) = |\phi_{z}(I)|.$$
\begin{definition}\label{shadow}{\rm
Let $z \in  \widehat{\Bbb C}$ and $I \subset \Bbb T$ be an arc segment. We say that $z$ is in the shadow of $I$ or shadowed by $I$ if either $z \in I$ or if $z \in \widehat{\Bbb C} \setminus \Bbb T$ with $\mu_{z}(I) \ge 2\pi /3$. }
\end{definition}

The following  lemma can be directly derived from the definitions and the reader shall easily provide a proof.

\begin{lemma}\label{tv}
Let $z \in  \widehat{\Bbb C}$ and $I \subset \Bbb T$ be an arc segment.  Then the following three properties hold,
\begin{itemize}
\item[1.] $z \in D(I)$ if and only if $z \in I$ or $\mu_{z}(I) > \pi$,
\item[2.] if $z \in D(I)$, then $z$ is in the shadow of $I$,
\item[3.] $z$ can be shadowed by at most three disjoint arc segments.
\end{itemize}
\end{lemma}

For a hyperbolic Riemann surface $X$, we use $\rho_{X}$ to denote the hyperbolic metric in $X$ and $d_{x}(\cdot, \cdot)$ denote the distance with respect to the hyperbolic metric $\rho_{U}$.

\begin{lemma}\label{lemma:higher order}
For any $d_{0} > 0$, there exists a $0< C_{0} < \infty$ depending
only on $d_{0}$ such that for any two distinct points $x, y \in
\Delta$,  the inequality
$$
\frac{\rho_{_{\Delta-\{y\}}}(x)}{ \rho_{_{\Delta}}(x)} \le 1 + C_{0}
e^{-2d_{\Delta}(x, y)}
$$
holds provided that $d_{\Delta}(x, y)> d_{0}$.
\end{lemma}

\begin{proof}
We need only to show that  $C_{0}$ can be taken to be a fixed
constant when $d_{\Delta}(x, y)$ is large enough.  To show this, it
is sufficient to consider the case that $y = 0$ and $x = 1 - \delta$
with $0 < \delta < 1$ small. By direct calculations, we have
$$
\rho_{\Delta-\{y\}}(x) = \frac{1}{(1- \delta)|\ln(1 - \delta)|} \:
\:  \hbox{ and }\: \rho_{\Delta}(x)= \frac{1}{\delta(1 - \delta/2)}.
$$
Note that for all  $0< \delta < 1$, we have
$$
(1 - \delta) |\ln(1 - \delta)| > (1- \delta) (\delta + \delta^{2}/2
+ \delta^{3}/3) > \delta(1 - \delta/2 - \delta^{2})
$$
and for all $0< \delta < 1/2$, we have
$$
\delta/2 + \delta^{2} < 1/2.
$$
Thus for all $0< \delta < 1/2$, we have
$$
\frac{\rho_{_{\Delta-\{y\}}}(x)}{ \rho_{_{\Delta}}(x)} <  1 +
\frac{\delta^{2}}{1 - \delta/2 - \delta^{2}} < 1 + 2 \delta^{2}.
$$
By a direct calculation, we get
$$
d_{\Delta}(x, y) = \ln \frac{2 - \delta}{\delta}.
$$
The lemma then follows since
$$
e^{-2 d_{\Delta}(x, y)}  = \frac{\delta^{2}}{(2 - \delta)^{2}} >
\delta^{2}/4.
$$
\end{proof}
\begin{lemma}\label{tilm}
There is a universal constant $0< C < \infty$ such that for any arc segment $I \subset \Bbb T$
with $|I| < 2\pi/3$, we have
$$
\frac{\rho_{W -\{0\}}}{\rho_{W}} \le e^{C |I|} \hbox{ on } D(I)
$$
where $W = \widehat{\Bbb C} - (\Bbb T - I)$.
\end{lemma}
\begin{proof}
For $0< \alpha < \pi$, let
\begin{equation}\label{gg-ll}
D_{\alpha}(I) = \{z\in W\:|\: d_{W}(z, I) < \ln\cot \frac{\alpha}{4}\}.
\end{equation}
By transforming the unit circle to the real line through a M\"{o}bius map, it follows that  $D_{\alpha}$ is the hyperbolic neighborhood of $I$ with the exterior angle being $\alpha$.  More precisely, $D_{\alpha}$ is a simply connected domain containing $I$ whose boundary is the union of two arc segments of Euclidean circles which are symmetric about the unit circle such that the exterior angle between $\partial D_{\alpha}$ and the unit circle is $\alpha$. To learn more details about the hyperbolic neighborhood in a slit plane, we refer the reader to  \cite{MS} ($\S5$ of Chapter VI ).
By the definition of $D(I)$, we get
$$
D(I) = D_{\pi/2}(I) =  \{z\in W\:|\: d_{W}(z, I) < \ln \cot \frac{\pi}{8}\}.
$$
It is not difficult to see that $0 \in \partial D_{|I|/2}(I)$. So we have
$$
d_{W}(0, D(I)) = \ln \cot \frac{|I|}{8} - \ln \cot \frac{\pi}{8}.
$$
 Since  $|I| \le 2 \pi/3$, we have $0<  \sin \frac{|I|}{8} < |I|/8$. We thus get
$$
\ln  \cot \frac{|I|}{8} > \ln  \frac{\cos \frac{\pi}{12}} {\frac{|I|}{8}} =\ln  \frac{8\cos \frac{\pi}{12}}{|I|}.$$
So for any  $z \in D(I)$, we have
\begin{equation}\label{eq-f}
d_{W}(0, z) > d_{W}(0, D(I)) \ge \ln  \frac{8\cos \frac{\pi}{12}}{|I|} - \ln \cot \frac{\pi}{8}.
\end{equation}
Since $|I| \le 2 \pi/3$, we have  $\cot \frac{|I|}{8} > \cot \frac{\pi}{12}$ and thus
\begin{equation}\label{eq-s}
d_{W}(0, z) > d_{W}(0, D(I)) \ge d_{0} = \ln \cot \frac{\pi}{12}- \ln \cot \frac{\pi}{8} > 0.
\end{equation}
For such $d_{0}$, let $C_{0}$ be the constant provided by Lemma~\ref{lemma:higher order}.  Then for any $z \in D(I)$, by Lemma~\ref{lemma:higher order} and (\ref{eq-f}), we have
$$
\frac{\rho_{W -\{0\}}(z)}{\rho_{W}(z)} \le 1 + C_{0} e^{-2 d_{W}(0, z)} < 1 + \frac{C_{0} \cdot \cot^{2} \frac{\pi}{8}}{64 \cos^{2} \frac{\pi}{12} } |I|^{2}.
$$ Since $|I| < 2 \pi/3$, we have $|I|^{2} < \frac{2 \pi}{3} |I|$. Take
$$
C = \frac{\pi\cdot  C_{0}  \cdot \cot^{2} \frac{\pi}{8}}{96 \cos^{2} \frac{\pi}{12}}.
$$
We then have for any $z \in D(I)$,
$$
\frac{\rho_{W -\{0\}}(z)}{\rho_{W}(z)} \le 1 + C |I| < e^{C|I|}.
$$
The proof of Lemma~\ref{tilm} is completed.
\end{proof}

\begin{rw}[\cite{BC}]
Let $R$ and $S$ be two hyperbolic Riemann surfaces and  $f: R \to S$ be a holomorphic map.   Then
$$
\frac{f^{*}\rho_{_{S}}}{\rho_{_{R}}} \le
\frac{f^{*}\rho_{_{S'}}}{\rho_{_{R'}}}\le 1.
$$
\end{rw}
For a detailed proof of the Relative Schwarz Lemma, we refer  the reader to \cite{BC}.

\begin{lemma}\label{nsd}
Let $C$ be the universal constant provided by Lemma~\ref{tilm}.  Let  $J \Subset I \subset \Bbb T$ be two arc segments and  $Z \subset \widehat{\Bbb C}$ be a finite set  such that no point in $Z$ is shadowed by $I$.    Let $\gamma$ be the core geodesic of the annulus $U = (\widehat{\Bbb C} - \Bbb T) \cup (I - J)$. Then
$$
\frac{l_{U-Z}(\gamma)}{l_{U}(\gamma)} \le \prod_{z \in Z} e^{C \mu_{z}(I)}.
$$
\end{lemma}
\begin{proof}
Let  $V =   \widehat{\Bbb C} - (\Bbb T - I)$. Let us label the points in $Z$ by $z_{1}, \cdots, z_{p}$. Let $Z_{0} = \emptyset$ and for  $1 \le k \le p$, let $Z_{k} = \{z_{1}, \cdots, z_{k}\}$.  Note that
$$
\frac{\rho_{U - Z}}{\rho_{U}} = \prod_{k=0}^{p-1} \frac{\rho_{U - Z_{k+1}}}{\rho_{U-Z_{k}}}.
$$
It follows from the Relative Schwarz Lemma that
$$
 \frac{\rho_{U - Z_{k+1}}}{\rho_{U-Z_{k}}}  \le  \frac{\rho_{U - \{z_{k+1}\}}}{\rho_{U}} \le \frac{\rho_{V - \{z_{k+1}\}}}{\rho_{V}}.
$$
So we finally have
\begin{equation}\label{pqw}
\frac{\rho_{U - Z}}{\rho_{U}} \le \prod_{z \in Z} \frac{\rho_{V - \{z\}}}{\rho_{V}}.
\end{equation}
Let $\phi_{z}$ be a M\"{o}bius map which preserves the unit circle and maps $z$ to $0$.  Then  $\phi_{z}(D(I)) =  D(\phi_{z}(I))$. Since $z$ is not shadowed by $I$, we have $|\phi_{z}(I)|< 2\pi/3$.  Note that $\phi_{z}(V) = \widehat{\Bbb C} - (\Bbb T - \phi_{z}(I))$. By Lemma~\ref{tilm}, we have
$$
\frac{\rho_{\phi_{z}(V) - \{0\}}}{\rho_{\phi_{z}(V)}} \le e^{C|\phi_{z}(I)|} = e^{C \mu_{z}(I)}   \hbox{  on  } D(\phi_{z}(I)) =  \phi_{z}(D(I)).
$$
Since the maps $\phi_{z}: V  \to \phi_{z}(V)$ and $\phi_{z}: V -\{z\} \to \phi_{z}(V) - \{0\}$ are  holomorphic isomorphisms,  it follows  that  $$\frac{\rho_{V - \{z\}}(w)}{\rho_{V}(w)}  = \frac{\rho_{\phi_{z}(V) - \{0\}}(\phi_{z}(w))}{\rho_{\phi_{z}(V)}(\phi_{z}(w))} \le e^{C \mu_{z}(I)} \hbox{  for all } w \in D(I). $$ This implies that
\begin{equation}\label{opw}
\frac{\rho_{V - \{z\}}}{\rho_{V}} \le  e^{C \mu_{z}(I)} \hbox{  on   } D(I)
\end{equation} From (\ref{pqw}) and (\ref{opw}) we have
$$
\frac{\rho_{U - Z}}{\rho_{U}} \le  \prod_{z \in Z} e^{C \mu_{z}(I)} \hbox{  on   } D(I).
$$
Note that   $\gamma \subset D(I)$  by Lemma~\ref{kgy}. We thus have
$$
\frac{\rho_{U - Z}}{\rho_{U}} \le  \prod_{z \in Z} e^{C \mu_{z}(I)} \hbox{  on   } \gamma.
$$
Lemma~\ref{nsd} then follows.
\end{proof}

Now let us prove Lemma~\ref{lemma:ex}.
\begin{proof}
Let $B \in {\mathbf{B}}_{\theta}^{m}$. In the beginning of $\S4.2$,
let $I^{k} = I^{k}_{B}$ and $J^{k} = J_{B}^{k}$ where $J_{B}^{k} \Subset I_{B}^{k} \subset \Bbb T$, $0 \le k \le N$,  are the arc segments  given in Lemma~\ref{lemma:ex}. for $0 \le k \le N$.  Let $Z$ be  the set of all the critical values of $B$ and $p = \# Z$. Then $$U_{k} = X_{B}^{k}, V_{k} = X_{B}^{k} - Z  \hbox{ and } l_{k} = l_{X^{k}_{B}}(\gamma^{k}_{B}) \hbox{ for }0 \le k \le N.$$    Since the number of critical values of $B$ is not more than the number of distinct critical points of $B$ which is not more than  $2m-2$,  it follows that $p \le 2 m -2$.

Let
$$
\Lambda_{1} =\{0 \le k \le N-1\:|\: I_{k} \hbox{  shadows at least one point of } Z\}.
$$
By the third assertion of Lemma~\ref{tv}, each point in $Z$ is shadowed by at most three intervals $I^{k}$. This implies that $$|\Lambda_{1}| \le 3p \le 6(m-1).$$
Let
$$
\Lambda_{2} = \{0 \le k \le N-1\:|\: k \notin \Lambda_{1}\}.
$$
Then
$$
\frac{l_{X^{N}_{B}}(\gamma^{N}_{B})} {l_{X^{0}_{B}}(\gamma^{0}_{B})} = \frac{l_{N}}{l_{0}} = \prod_{k=0}^{N-1} \frac{l_{k+1}}{l_{k}} = \bigg{(}\prod_{k\in\Lambda_{1}} \frac{l_{k+1}}{l_{k}}\bigg{)}\cdot \bigg{(} \prod_{k\in\Lambda_{2}} \frac{l_{k+1}}{l_{k}}\bigg{)}.
$$

$\bold{Claim\: 1.}$
\begin{equation}\label{claim-1}
\frac{l_{k+1}}{l_{k}}  \le m(2m -1)  \hbox{ for every } k \in \Lambda_{1}.
\end{equation}
Let us prove the Claim 1.  Let $k \in \Lambda_{1}$.  Let $\xi_{B}^{k}$ be one of the shortest
simple closed geodesics in  $V_{k}$ separating $J^{k}$ and $\Bbb T - I^{k}$. By the minimal property of $\xi_{B}^{k}$, it follows that $\xi_{B}^{k}$ is symmetric about the unit circle.  In particular, the unit circle and $\xi_{B}^{k}$ have two intersection points where they cross perpendicularly. Let $a_{k}$ and $b_{k}$ be the two intersection points. Let $a_{k}'$ and $b_{k}'$ be the two points in the unit circle such that $B(a_{k}') = a_{k}$ and $B(b_{k}') = b_{k}$.  Let $W_{k+1}$ be the component of $B^{-1}(V_{k})$ which contains $a_{k}'$. It is clear that $W_{k+1} \subset U_{k+1}$ and the map $B: W_{k+1} \to V_{k}$ is a holomorphic covering map. Let $\eta_{B}^{k+1}$ be the simple closed geodesic in $W_{k+1}$ such that $a_{k}' \in \eta_{B}^{k+1}$ and $B(\eta_{B}^{k+1}) = \xi_{B}^{k}$.  Then $\eta_{B}^{k+1}$ crosses the unit circle at $a_{k}'$ perpendicularly. It follows that $\eta_{B}^{k+1}$ and the unit circle must  have at least two intersection points. Since $\xi_{B}^{k}$ intersects the unit circle at exactly two points $a_{k}$ and $b_{k}$ and the map $B|\Bbb T: \Bbb T\to \Bbb T$ is a homeomorphism, $\eta_{B}^{k+1}$ and the unit circle  have exactly  two intersection points, $a_{k}'$ and $b_{k}'$. Since $\xi_{B}^{k}$ crosses the unit circle perpendicularly, $\eta_{B}^{k+1}$ crosses the unit circle perpendicularly also. In particular, $\eta_{B}^{k+1}$ separates $\Bbb T - I^{k+1}$ and $J^{k+1}$. Thus we have
$$
l_{X^{k+1}_{B}}(\gamma^{k+1}_{B}) \le l_{X^{k+1}_{B}}(\eta^{k+1}_{B}).
$$
Since $W_{k+1} \subset U_{k+1}$ we have $\rho_{_{W_{k+1}}} \ge \rho_{_{U_{k+1}}}$.  So we have
$$
l_{X^{k+1}_{B}}(\eta^{k+1}_{B})\le  l_{W_{k+1}}(\eta_{B}^{k+1}).
$$
Since $B: W_{k+1} \to V_{k}$ is a holomorphic covering map and the degree of $B$ is $m$,  it follows that
$$
l_{W_{k+1}}(\eta_{B}^{k+1}) \le m \cdot l_{V_{k}}(\xi_{B}^{k}).
$$
By the choice of $\xi_{B}^{k}$ and Corollary~\ref{cry-1}, we have
$$
l_{V_{k}}(\xi_{B}^{k})  = l_{k}' \le (p+1) \cdot  l_{U_{k}}(\gamma_{B}^{k}) = (p+1) \cdot l_{k}\le (2m-1) \cdot l_{k}.
$$
 This, together with the above three inequalities,  implies that
$$
l_{k+1} = l_{X^{k+1}_{B}}(\gamma^{k+1}_{B}) \le m(2m-1)  \cdot l_{k}.
$$
This proves (\ref{claim-1}) and the Claim 1  has been proved.

Let $0< C < \infty$ be the universal constant in Lemma~\ref{nsd}.

$\bold{Claim\: 2.}$
\begin{equation}\label{claim-2}
\frac{l_{k+1}}{l_{k}} \le \prod_{z \in Z} e^{C \mu_{z}(I^{k})} \hbox{ for every k} \in \Lambda_{2}.
\end{equation}

Let us prove the Claim 2.   Let $k \in \Lambda_{2}$.  By Lemma~\ref{kgy}, we have $\gamma_{B}^{k} \subset D(I^{k})$. Since $I^{k}$ does not shadow any point in $Z$, it follows that $D(I^{k})$ does not intersect $Z$. This implies that $\gamma_{B}^{k}$ does not contain any point in $Z$. We thus have $\gamma_{B}^{k} \subset V_{k}$. Let $\xi_{B}^{k}$ be the unique simple closed geodesic in $V_{k}$ which is homotopic to
$\gamma_{B}^{k}$ in $V_{k}$.  Then $\xi_{B}^{k}$ separates $\Bbb T - I^{k}$ and $J^{k}$, and moreover,
\begin{equation}\label{kk-r}
l_{V_{k}}(\xi_{B}^{k}) \le l_{V_{k}}(\gamma_{B}^{k}).
\end{equation}
Since $\gamma_{B}^{k}$ and $V_{k}$ are symmetric about the unit circle,   $\xi_{B}^{k}$ is symmetric about the unit circle also. In particular, the unit circle and $\xi_{B}^{k}$ have two intersection points where they cross perpendicularly.
Now let $W_{k+1}$ and $\eta_{B}^{k+1}$ be as in the proof of the Claim 1. By the same argument as before, it follows that $\eta_{B}^{k+1}$ separates $\Bbb T - I^{k+1}$ and $J^{k+1}$, and the map $B: W_{k+1} \to V_{k}$ is a holomorphic covering map.
Let $\Omega$ be the component of  $\widehat{\Bbb C} - \gamma_{B}^{k}$ which contains $J^{k}$. Since $D(I^{k})$ does not intersect the set $Z$ and  since $\gamma_{B}^{k} \subset  D(I^{k})$ by Lemma~\ref{kgy}, it follows that \begin{equation}\label{cors1}
\Omega \cap Z = \emptyset.
\end{equation}
Let $\tilde{\Omega}$  be the component of  $\widehat{\Bbb C} - \xi_{B}^{k}$ which contains $J^{k}$. Since $\xi_{B}^{k}$ is homotopic to
$\gamma_{B}^{k}$ in $V_{k}$,  from (\ref{cors1}) we get   $$\tilde{\Omega} \cap Z = \emptyset.
 $$ This implies that $\tilde{\Omega}$ contains no critical value of $B$.   It follows   that the covering degree of the map
$$B|\eta_{B}^{k+1}: \eta_{B}^{k+1} \to \xi_{B}^{k}$$ is one. We thus have
\begin{equation}\label{lt}
l_{W_{k+1}}(\eta_{B}^{k+1}) =   l_{V_{k}}(\xi_{B}^{k}).
\end{equation}
Since $W_{k+1} \subset U_{k+1} =X^{k+1}_{B} $ we have $\rho_{_{W_{k+1}}} \ge \rho_{_{U_{k+1}}} = \rho_{X^{k+1}_{B}}$, and thus $$l_{X^{k+1}_{B}}(\eta^{k+1}_{B})\le  l_{W_{k+1}}(\eta_{B}^{k+1}).$$  This, together with (\ref{kk-r}) and (\ref{lt}), implies that $l_{X^{k+1}_{B}}(\eta^{k+1}_{B})\le l_{V_{k}}(\gamma_{B}^{k})$. Since $l_{X^{k+1}_{B}}(\gamma^{k+1}_{B}) \le l_{X^{k+1}_{B}}(\eta^{k+1}_{B})$, we thus have
\begin{equation}\label{cc-u}
l_{k+1} = l_{X^{k+1}_{B}}(\gamma^{k+1}_{B}) \le l_{X^{k+1}_{B}}(\eta^{k+1}_{B})\le l_{V_{k}}(\gamma_{B}^{k}).\end{equation} By Lemma~\ref{nsd}, we have
\begin{equation}\label{cc-v}
\frac{l_{V_{k}}(\gamma_{B}^{k})}{l_{k}} = \frac{l_{V_{k}}(\gamma_{B}^{k})}{l_{U_{k}}(\gamma_{B}^{k})} \le \prod_{z \in Z} e^{C \mu_{z}(I^{k})}.
\end{equation}
From (\ref{cc-u}) and (\ref{cc-v}) we have
$$
 \frac{l_{k+1}}{l_{k}} \le  \prod_{z \in Z} e^{C \mu_{z}(I^{k})}.
$$
This proves the Claim 2.

From Claims 1 and 2 we have
$$
\frac{l_{X^{N}_{B}}(\gamma^{N}_{B})} {l_{X^{0}_{B}}(\gamma^{0}_{B})} = \bigg{(}\prod_{k\in\Lambda_{1}} \frac{l_{k+1}}{l_{k}}\bigg{)}\cdot \bigg{(} \prod_{k\in\Lambda_{2}} \frac{l_{k+1}}{l_{k}}\bigg{)} \le \big{(}m(2m -1)\big{)}^{6(m-1)} \prod_{k \in \Lambda_{2}} e^{C \mu_{z}(I^{k})}
$$ Since  $$ \sum_{k \in \Lambda_{2}} \mu_{z}(I^{k}) \le 2 \pi,$$ we finally have
$$
\frac{l_{X^{N}_{B}}(\gamma^{N}_{B})} {l_{X^{0}_{B}}(\gamma^{0}_{B})} \le e^{2 \pi C} \big{(}m(2m -1)\big{)}^{6(m-1)}.
$$
This completes the proof of Lemma~\ref{lemma:ex}.
\end{proof}

%%%%%%%%%%%%%%%%%%%%%%%%%%%%%%%%%%%%%%%%%%%%%%%%%%%%%%%%%%%%%%%%%%%%%%%%%%%%
%%%%%%%%%%%%%%%%%%%%%%%%%%%%%%%%%%%%%%%%%%%%%%%%%%%%%%%%%%%%%%%%%%%%%%%%%%%%
%%%%%%%%%%%%%%%%%%%%%%%%%%%%%%%%%%%%%%%%%%%%%%%%%%%%%%%%%%%%%%%%%%%%%%%%%%%%
%%%%%%%%%%%%%%%%%%%%%%%%%%%%%%%%%%%%%%%%%%%%%%%%%%%%%%%%%%%%%%%%%%%%%%%%%%%%
%%%%%%%%%%%%%%%%%%%%%%%%%%%%%%%%%%%%%%%%%%%%%%%%%%%%%%%%%%%%%%%%%%%%%%%%%%%%
%%%%%%%%%%%%%%%%%%%%%%%%%%%%%%%%%%%%%%%%%%%%%%%%%%%%%%%%%%%%%%%%%%%%%%%%%%%%

\subsection{Proof of Theorem B} All the arguments used in  this
section are standard. The readers may find them in several previous
literatures, for instance, see \cite{dFdM}, \cite{H}, and
\cite{Pe2}.

Let $B \in {\mathbf{B}}_{\theta}^{m}$ be a centered Blaschke product.  Recall that  $h_{B}: {\Bbb T} \to {\Bbb T}$ is
the circle homeomorphism such that $B|{\Bbb T} = h_{B}^{-1} \circ
R_{\theta}\circ h_{B}$ and $h_{B}(1) = 1$. Now it is  sufficient to prove that there exists an $1 < M(m,
\theta) < \infty$ depending only on $m$ and $\theta$ such that
$h_{{B}}: {\Bbb T} \to {\Bbb T}$ is an $M(m,
\theta)$-quasisymmetric circle homeomorphism.  Before that let us introduce some notations and terminologies first.

 Let $I_{1}$ and  $I_{2}$ be two arc segments in ${\Bbb
T}$. Let $L > 1$. We say $I_{1}$ and $I_{2}$ are $L$-comparable if
$$|I_{2}|/L< |I_{1}|< L|I_{2}|.$$  Let $a, b \in \Bbb T$ be two distinct points.
Recall that we use
$[a, b]$ to denote the arc segment in $\Bbb T$ which connects $a$
and $b$ anticlockwise and  $|[a, b]|$ to denote the Euclidean length
of $[a, b]$.  For an arc segment $[a, b]$ with  $\big{|}{h}_{B}([a, b]) \big{|} \ne \pi$,
let us use $\langle a, b \rangle$ to denote $[a, b]$ if
$\big{|}{h}_{B}([a, b]) \big{|} <  \pi$, and denote $[b, a]$ if
$\big{|}{h}_{B}([a, b]) \big{|} > \pi$.

Let $\theta = [a_{1}, \cdots, a_{n}, \cdots]$. Let $q_{0} = 1$,
$q_{1} = a_{1}$, and $q_{n+1} = q_{n-1} + a_{n+1}q_{n}$ for all $n
\ge 1$. For $x > 0$, let $\{x\}$ denote the fraction part of $x$.
For $n \ge 0$, let  $\langle q_{n} \theta \rangle$ denote  $\{q_{n}
\theta \}$ if $n$ is even and $1- \{q_{n} \theta \}$ if $n$ is odd.
\begin{lemma}\label{uct}
There exists an $L_{0}
\ge 2 $ independent of $\theta$,  such that for all $n \ge
L_{0}$, the following inequality holds,
\begin{equation}\label{univ-c}
 \langle q_{n} \theta \rangle < 1/6.
\end{equation}
\end{lemma}
\begin{proof}
For $n \ge 0$, let $p_{n}/q_{n}$ be the $n$th continued fraction. Let
$$
\delta_{n} = \frac{p_{n}}{q_{n}} - \theta.
$$
It follows that $|\delta_{n}| < 1/q_{n} q_{n+1}$ (for instance, see \cite{Khi}). This implies that
$$
 \langle q_{n} \theta \rangle = |q_{n} \delta_{n}| < 1/q_{n+1}.
$$
Note that $q_{0} = 1$,  $q_{1} \ge 1$ and $q_{n+2} \ge q_{n} + q_{n+1}$ for all $n \ge 0$. The lemma then follows by taking $L_{0} = 5$.
\end{proof}

As an immediate consequence of Lemma~\ref{uct}, we have
\begin{corollary}\label{nus}
Let $L_{0}$ be the constant in Lemma~\ref{uct}. Then for any $n \ge L_{0}$ and   any $z \in \Bbb T$,  we
have
$$
\langle {R_{\theta}}^{-q_{n}}(z), {R_{\theta}}^{2q_{n}}(z) \rangle =
\langle {R_{\theta}}^{-q_{n}}(z), z \rangle  \cup \langle z,
{R_{\theta}}^{q_{n}}(z) \rangle  \cup \langle {R_{\theta}}^{q_{n}}(z),
{R_{\theta}}^{2q_{n}}(z) \rangle.
$$
\end{corollary}

\begin{lemma}\label{bl}
Suppose that $n \ge L_{0}$. Let $z \in \Bbb T$. Then the following
two assertions hold.
\begin{itemize}
\item[1.] Let $I = \langle R_{\theta}^{-q_{n}}(z),
R_{\theta}^{2q_{n}}(z) \rangle$.
 Then  $\{R_{\theta}^{-k}(I)\: \big{|}
\: 0 \le k \le q_{n-2}-1\}$  is a disjoint family.
\item[2.] Let $I = \langle z,
R_{\theta}^{q_{n}}(z) \rangle$.  Then $\Bbb T \subset
\bigcup_{k=0}^{q_{n} + q_{n+1} -1} R_{\theta}^{-k}(I)$.
\end{itemize}
\end{lemma}
\begin{proof}
The second assertion  is standard, for instance, see \cite{dFdM},
\cite{H}, and \cite{Pe2}. Let us prove the first assertion only. Let
us prove it by contradiction. Suppose it were not true. Then there
exists a $0 < k < q_{n-2}$ such that
$$
R_{\theta}^{-k}(z) \in \langle R_{\theta}^{-3q_{n}}(z),
R_{\theta}^{3q_{n}}(z) \rangle.
$$
It is clear that $R_{\theta}^{-k}(z) \notin \langle
R_{\theta}^{-q_{n}}(z), R_{\theta}^{q_{n}}(z) \rangle$ by the property
of the closest returns. Then we have the following four cases.

In the first case, $R_{\theta}^{-k}(z) \in \langle
R_{\theta}^{-3q_{n}}(z),R_{\theta}^{-2q_{n}}(z) \rangle$. Let $\xi =
R_{\theta}^{-3q_{n}}(z)$. Then $R_{\theta}^{3q_{n}-k}(\xi) \in \langle
\xi, R_{\theta}^{q_{n}}(\xi) \rangle$. We then must have $3q_{n} - k
= q_{n} + q_{n+1}$. Since $q_{n+1} = q_{n-1} + a_{n+1} q_{n}$, it
follows that $a_{n+1} = 1$. So $k = q_{n} - q_{n-1} \ge q_{n-2}$.
This is a contradiction.

In the second case, $R_{\theta}^{-k}(z) \in \langle
R_{\theta}^{-2q_{n}}(z), R_{\theta}^{-q_{n}}(z) \rangle$. Let $\xi =
R_{\theta}^{-q_{n}}(z)$. Then $R_{\theta}^{q_{n}-k}(\xi) \in \langle
R_{\theta}^{-q_{n}}(\xi), \xi  \rangle$. Since $0 < q_{n} - k <
q_{n}$, this is impossible.

In the third case, $R_{\theta}^{-k}(z) \in \langle
R_{\theta}^{q_{n}}(z),R_{\theta}^{2q_{n}}(z) \rangle$. Let $\xi =
R_{\theta}^{-k}(z)$. Then $R_{\theta}^{2q_{n}+k}(\xi) \in \langle \xi,
R_{\theta}^{q_{n}}(\xi) \rangle$. Since  $q_{n} < 2q_{n}+k = q_{n} +
q_{n} + k < q_{n} + q_{n+1}$, this is impossible.

In the last case, $R_{\theta}^{-k}(z) \in \langle
R_{\theta}^{2q_{n}}(z), R_{\theta}^{3q_{n}}(z) \rangle$. Let $\xi =
R_{\theta}^{-k}(z)$. Then $R_{\theta}^{3q_{n}+k}(\xi) \in \langle \xi,
R_{\theta}^{q_{n}}(\xi) \rangle$. Since $0< k< q_{n-2}$,  we  must
have $q_{n} < 3q_{n} + k <  q_{n} + 2q_{n+1}$.
We claim that
$3q_{n} + k =  q_{n} + q_{n+1}$. Let us prove the claim. Assume that the claim were not true.
There are two cases. In the first case,  we have  $3q_{n} + k =  q_{n} + l$ with $ 2q_{n} < l < q_{n+1}$.  Then by the property of the closest returns,  we have
$|\langle R_{\theta}^{q_{n}}(\xi),
R_{\theta}^{q_{n} + l}(\xi) \rangle| = |\langle \xi,
R_{\theta}^{l}(\xi) \rangle| > |\langle \xi,
R_{\theta}^{q_{n}}(\xi) \rangle|$. This is a contradiction with $R_{\theta}^{q_{n} + l}(\xi) =  R_{\theta}^{3q_{n}+k}(\xi) \in \langle \xi,
R_{\theta}^{q_{n}}(\xi) \rangle$. In the second case, we have   $3q_{n} + k =  q_{n} + q_{n+1} +  l'$ with some $l' > 0$. Since   $0< k < q_{n-2}$, it follows that $l' = 2q_{n} + k - q_{n+1} < 2q_{n} + q_{n-2} - q_{n+1} < q_{n}$. Since both $R_{\theta}^{q_{n} + q_{n+1}}(\xi)$ and $R_{\theta}^{3q_{n} + k}(\xi)$ belong to $ \langle \xi,
R_{\theta}^{q_{n}}(\xi) \rangle$, it follows that $|\langle \xi,
R_{\theta}^{l'}(\xi) \rangle| = |\langle R_{\theta}^{q_{n} + q_{n+1}}(\xi),
R_{\theta}^{3q_{n} + k}(\xi) \rangle| <  |\langle \xi,
R_{\theta}^{q_{n}}(\xi) \rangle|$. This is again impossible. Thus  the claim has been proved and
we must have $3q_{n} + k =  q_{n} + q_{n+1}$.

By the claim we just proved, we have $q_{n+1} =  2q_{n} + k$.
Since $q_{n+1} = q_{n-2} + a_{n+1} q_{n}$ and $0< k <q_{n-2}$, we
get $a_{n+1} = 1$. This implies that $q_{n-2} = q_{n} + k$. This is
impossible. The proof of the lemma is completed.
\end{proof}

Let $L_{0} > 0$ be the universal constant provided in Lemma~\ref{uct}.
\begin{lemma}\label{lemma:DF}
There exists a $1 < J(m, \theta) < \infty$ depending only on $m$ and
$\theta$ such that for every centered Blaschke $B \in {\mathbf{B}}_{\theta}^{m}$, any $n
\ge L_{0}$, and  any $z \in {\Bbb T}$, the following two inequalities
hold,
\begin{equation}\label{cp-1}
1/ J(m, \theta) \le \frac{\big{|}\langle {B}^{-q_{n}}(z), z
\rangle\big{|}}{\big{|}\langle z, {B}^{q_{n}}(z)
\rangle\big{|}} \le  J(m, \theta)
\end{equation}
and
\begin{equation}\label{cp-2}
1/ J(m, \theta) \le \frac{\big{|}\langle {B}^{q_{n+1}}(z), z
\rangle\big{|}}{\big{|}\langle z, {B}^{q_{n}}(z)
\rangle\big{|}} \le  J(m, \theta).
\end{equation}
\end{lemma}

\begin{proof}

Let  $n \ge L_{0}$. Take $z_{0} \in {\Bbb T}$ such that
$$
\big{|}\langle z_{0}, {B}^{q_{n}}(z_{0}) \rangle \big{|} =
\min_{z \in \Bbb T} \big{|}\langle z, {B}^{q_{n}}(z) \rangle
\big{|}.
$$
It follows that
\begin{equation}\label{cr-1}
C(\langle {B}^{-q_{n}}(z_{0}), {B}^{2q_{n}}(z_{0})
\rangle, \langle z_{0}, {B}^{q_{n}}(z_{0}) \rangle) < 3.
\end{equation}
 Since $\theta$ is of
bounded type, there is an  integer $0< \tau(\theta) < \infty$
depending only on $\theta$ such that
\begin{equation}\label{qqb-1}
q_{n} < \tau(\theta) q_{n-2}
\end{equation}
 for all $n \ge 2$. By
the first assertion of Lemma~\ref{bl}, it follows that for any integer $0< N \le
5q_{n}$, the family $$\{\langle {B}^{-q_{n}-k}(z_{0}), {B}^{2q_{n}-k}(z_{0})
\rangle\:\big{|} \: 0 \le k \le
N\}$$ can be divided into at most $5\tau(\theta)$ disjoint
sub-families. By (\ref{cr-1}) and by applying Lemma~\ref{Swiatek-D} successively at most  $5\tau(\theta)$ times,  we get a constant $0< P_{1}(m, \theta) < \infty$ depending only on $m$  and $\theta$
such that the following inequality
\begin{equation}\label{cr-2o}
C(\langle {B}^{-(l+1)q_{n}}(z_{0}),
{B}^{(2-l)q_{n}}(z_{0}) \rangle, \langle
{B}^{-lq_{n}}(z_{0}), {B}^{(1-l)q_{n}}(z_{0}) \rangle) <
P_{1}(m, \theta)
\end{equation}
holds for $0 \le l \le 5$.

We claim that there exists a $0<
P_{2}(m, \theta) < \infty$ depending only on $m$  and $\theta$ such
that any two of the following six arc segments
\begin{equation}\label{sac}
\langle {B}^{-lq_{n}}(z_{0}),{B}^{(1-l)q_{n}}(z_{0})
\rangle, \: 0 \le l \le 5,
\end{equation}
are $P_{2}(m, \theta)$-comparable. Let us prove the claim. It suffices  to prove that among these six arc segments, any two adjacent ones are $P_{1}(m,\theta)$-comparable. Let us  prove this only for the pair of adjacent arc segments
$$
\langle z_{0},{B}^{q_{n}}(z_{0})
\rangle \hbox{  and } \langle {B}^{-q_{n}}(z_{0}),z_{0} \rangle.
$$  The same way can be used for the other four pairs of adjacent arc segments.   By taking $l = 0$ in  (\ref{cr-2o}) we get
$$
C(\langle {B}^{-q_{n}}(z_{0}),
{B}^{2q_{n}}(z_{0}) \rangle, \langle
z_{0}, {B}^{q_{n}}(z_{0}) \rangle) <
P_{1}(m, \theta).
$$
This implies that
\begin{equation}\label{pp-cr}
\frac{|\langle z_{0},{B}^{q_{n}}(z_{0})
\rangle|}   {  |\langle {B}^{-q_{n}}(z_{0}),z_{0} \rangle|   } < P_{1}(m, \theta).
\end{equation}
By taking $l = 1$ in  (\ref{cr-2o}) we get
$$
C(\langle {B}^{-2q_{n}}(z_{0}),
{B}^{q_{n}}(z_{0}) \rangle, \langle
{B}^{-q_{n}}(z_{0}), z_{0}  \rangle) <
P_{1}(m, \theta).
$$
This implies that
\begin{equation}\label{bb-cr}
\frac{|\langle {B}^{-q_{n}}(z_{0}),z_{0} \rangle| }   {  |\langle z_{0},{B}^{q_{n}}(z_{0})
\rangle|  } < P_{1}(m, \theta).
\end{equation}
From (\ref{pp-cr}) and (\ref{bb-cr}) it follows that the two adjacent arc segments $\langle z_{0},{B}^{q_{n}}(z_{0})
\rangle$  and $\langle {B}^{-q_{n}}(z_{0}),z_{0} \rangle$ are $P_{1}(m, \theta)$-comparable. The same way can be used to prove the other four adjacent arc segments are also $P_{1}(m, \theta)$-comparable. The claim then follows by taking $P_{2}(m, \theta) = P_{1}^{5}(m, \theta)$.

Let $$l_{0} = \big{|}\langle z_{0}, {B}^{q_{n}}(z_{0}) \rangle \big{|}.$$  By the choice of $z_{0}$, it follows that $l_{0}$ is  the minimum of
the length of the  six intervals in (\ref{sac}). By the Claim we proved above, we have
\begin{equation}\label{ah}
P_{2}(m, \theta)^{-1} \cdot l_{0} \le |\langle {B}^{-lq_{n}}(z_{0}),{B}^{(1-l)q_{n}}(z_{0})
\rangle| \le P_{2}(m, \theta) \cdot l_{0}, \: 0 \le l \le 5.
\end{equation}

For any $z\in \Bbb T$, it follows from the second assertion of Lemma~\ref{bl}
that there is an $0 \le i < q_{n}+ q_{n+1}$ such that
${B}^{i}(z) \in  \langle {B}^{-5q_{n}}(z_{0}),
{B}^{-4q_{n}}(z_{0}) \rangle$. We then have the following two
cases.

In the first case, there is some $1 \le j \le 3$ such that
$$
\big{|}\langle {B}^{i+jq_{n}}(z), {B}^{i+(j+1)q_{n}}(z)
\rangle \big{|} < l_{0}/2.
$$
This implies
\begin{equation}\label{cr-2}
C(\langle {B}^{i+(j-1)q_{n}}(z),
{B}^{i+(j+2)q_{n}}(z)\rangle, \langle {B}^{i+jq_{n}}(z),
{B}^{i+(j+1)q_{n}}(z) \rangle) < 3.
\end{equation}
Since $0\le i < q_{n} + q_{n+1}$ and $1 \le j \le 3$, by (\ref{qqb-1}) we have $$0< i + jq_{n} < 4q_{n} + q_{n+1} < (4\tau(\theta) + \tau(\theta)^{2}) q_{n-2}.$$
By (\ref{cr-2}) and the first assertion
of Lemma~\ref{bl}, and by applying   Lemma~\ref{Swiatek-D} successively at most $(4\tau(\theta) + \tau(\theta)^{2})$ times,  we get a constant $P_{3}(m, \theta)
> 0$ depending only on $m$ and $\theta$ such that
\begin{equation}\label{cr-3}
C(\langle {B}^{-q_{n}}(z),{B}^{2q_{n}}(z)\rangle,
\langle z, {B}^{q_{n}}(z) \rangle) < P_{3}(m, \theta).
\end{equation}

In the second case, we have
$$
\big{|}\langle {B}^{i+jq_{n}}(z), {B}^{i+(j+1)q_{n}}(z)
\rangle \big{|} \ge l_{0}/2
$$
for each $j =1, 2, 3$.  This, together with (\ref{ah}),  implies that  there exists a $0< P_{4}(m,
\theta) < \infty$ depending only on $m$  and $\theta$ such that
\begin{equation}\label{cr-4}
C(\langle {B}^{i+q_{n}}(z), {B}^{i + 4q_{n}}(z)\rangle,
\langle {B}^{i + 2q_{n}}(z), {B}^{i +3q_{n}}(z) \rangle)
< P_{4}(m, \theta).
\end{equation} Since $0< i + 2q_{n} < 3 q_{n} + q_{n+1} < (3 \tau(\theta) + \tau(\theta)^{2}) q_{n-2}$,
By (\ref{cr-4}) and the
first assertion of Lemma~\ref{bl}, and by applying  Lemma~\ref{Swiatek-D} successively at most $(3 \tau(\theta) + \tau(\theta)^{2})$ times,   we get a constant
$0< P_{5}(m, \theta) < \infty$ depending only on $m$ and $\theta$
such that
\begin{equation}\label{cr-5}
C(\langle {B}^{-q_{n}}(z), {B}^{2q_{n}}(z)\rangle,
\langle z, {B}^{q_{n}}(z) \rangle) < P_{5}(m, \theta).
\end{equation}
Let $P_{6}(m, \theta) = \max\{P_{3}(m, \theta), P_{5}(m, \theta)\}$. From (\ref{cr-3}) and (\ref{cr-5}) it follows that in both the cases, the following inequality holds,
\begin{equation}\label{dtcr}
C(\langle {B}^{-q_{n}}(z), {B}^{2q_{n}}(z)\rangle,
\langle z, {B}^{q_{n}}(z) \rangle) < P_{6}(m, \theta).
\end{equation} Since (\ref{dtcr}) holds for  an arbitrary $z \in \Bbb T$, by considering the point $B^{-q_{n}}(z)$, we get
\begin{equation}\label{dtcr'}
C(\langle {B}^{-2q_{n}}(z), {B}^{q_{n}}(z)\rangle,
\langle {B}^{-q_{n}}(z), z \rangle) < P_{6}(m, \theta).
\end{equation}
From (\ref{dtcr}) we have $|\langle z, {B}^{q_{n}}(z) \rangle| < P_{6}(m, \theta)|\langle {B}^{-q_{n}}(z), z\rangle|$. From (\ref{dtcr'}) we have $|\langle {B}^{-q_{n}}(z), z\rangle| < P_{6}(m, \theta)|\langle z, {B}^{q_{n}}(z) \rangle|$.
This implies that  for any $z \in \Bbb T$, the inequality
\begin{equation}\label{ff}
1/ P_{6}(m, \theta) \le \frac{\big{|}\langle {B}^{-q_{n}}(z),
z \rangle\big{|}}{\big{|}\langle z, {B}^{q_{n}}(z)
\rangle\big{|}} \le  P_{6}(m, \theta)
\end{equation}
holds  for all $n \ge L_{0}$.  This proves the first assertion of Lemma~\ref{lemma:DF} by taking $J(m, \theta) = P_{6}(m, \theta)$.

Now let us prove the second assertion of Lemma~\ref{lemma:DF}.  Note that
$$
\langle z, {B}^{-q_{n+1}}(z) \rangle  \subset \langle z,
{B}^{q_{n}}(z) \rangle,
$$
so from (\ref{cp-1}), we have
$$
\big{|} \langle {B}^{ q_{n+1}}(z), z \rangle \big{|} \le J(m,
\theta)\big{|}\langle z,  {B}^{- q_{n+1}}(z) \rangle \big{|} <
J(m, \theta)\big{|}\langle z,  {B}^{q_{n}}(z) \rangle \big{|},
$$
and this implies the right hand of (\ref{cp-2}). To prove the left
hand, Note that
$$
\langle z, {B}^{q_{n}}(z) \rangle \subset \bigcup_{0 \le i \le
b(\theta)} \langle {B}^{-iq_{n+1}}(z),
{B}^{-(i+1)q_{n+1}}(z) \rangle,
$$
where $b(\theta) = \sup\{a_{n}\}$. This implies that
$$
\big{|} \langle z, {B}^{q_{n}}(z) \rangle  \big{|}  \le
\sum_{0 \le i\le b(\theta)} \big{|}\langle {B}^{-iq_{n+1}}(z),
{B}^{-(i+1)q_{n+1}}(z) \rangle \big{|}.
$$
For $0 \le i \le b(\theta)$, by applying (\ref{cp-1}), we have
$$
\big{|}\langle {B}^{-iq_{n+1}}(z),
{B}^{-(i+1)q_{n+1}}(z) \rangle \big{|} \le J(m,
\theta)^{i+1}\big{|}\langle {B}^{q_{n+1}}(z), z \rangle
\big{|}.
$$
 Therefore, we get
$$
\frac{\big{|}\langle z, {B}^{q_{n}}(z)
\rangle\big{|}}{\big{|}\langle {B}^{q_{n+1}}(z), z
\rangle\big{|}} \le \sum_{0 \le i\le b(\theta)}J(m, \theta)^{i+1}.
$$
This proves the second assertion of the Lemma by modifying $J(m,
\theta)$. This completes the proof of Lemma~\ref{lemma:DF}.
\end{proof}

%%%%%%%%%%%%%%%%%%%%%%%%%%%%%%%%%%%%%%%%%%%%%%%%%%%%%%%%%%%%%%%%%%%%%%%%%%
%%%%%%%%%%%%%%%%%%%%%%%%%%%%%%%%%%%%%%%%%%%%%%%%%%%%%%%%%%%%%%%%%%%%%%%%%%
%%%%%%%%%%%%%%%%%%%%%%%%%%%%%%%%%%%%%%%%%%%%%%%%%%%%%%%%%%%%%%%%%%%%%%%%%%
%%%%%%%%%%%%%%%%%%%%%%%%%%%%%%%%%%%%%%%%%%%%%%%%%%%%%%%%%%%%%%%%%%%%%%%%%%

Now let us prove Theorem B.
Let $L_{0}> 0$ be the integer in Lemma~\ref{uct}. Take an
arbitrary $z \in \Bbb T$ and an arbitrary $0< \delta< 2 \pi$.

First let us assume that one of $\langle z,
{B}^{q_{_{L_{0}}}}(z) \rangle$ and  $\langle z,
{B}^{q_{_{L_{0}+1}}}(z) \rangle$ is contained either in
$[e^{-i \delta} z, z]$ or in $[z, e^{i\delta}z]$. With this assumption  let us show that
there exists an $1 < M_{1}(m, \theta)$ depending on only on $m$ and
$\theta$ such that
\begin{equation}\label{MM-1}
M_{1}(m, \theta)^{-1} < \frac{\big{|}h_{{B}}([z, e^{i \delta}
z])\big{|}}{ \big{|}h_{{B}}([ e^{-i \delta} z, z])\big{|}}<
M_{1}(m, \theta).
\end{equation}
 Without loss of generality, let us assume that
\begin{equation}\label{assu-10}
\langle z, {B}^{q_{_{L_{0}}}}(z) \rangle = [z,
{B}^{q_{_{L_{0}}}}(z)] \subset [z, e^{i\delta}z].
\end{equation}
 Since $\theta$ is of bounded type, by
Lemma~\ref{lemma:DF},  there is an integer $N_{1}(m, \theta)$
depending only on $m$ and $\theta$ such that
$$
\big{|} \langle {B}^{q_{L_{0}+1+2N_{1}(m, \theta)}}(z), z
\rangle  \big{|} \le \big{|} \langle z, {B}^{q_{L_{0}}}(z)
\rangle \big{|}.
$$
We thus have
\begin{equation}\label{assu-20}
\langle {B}^{q_{L_{0}+1+2N_{1}(m, \theta)}}(z), z \rangle
\subset [e^{-i \delta} z, z].
\end{equation}
From (\ref{assu-10}) we have
\begin{equation}\label{inc-7}
\langle q_{L_{0}} \theta \rangle  \cdot  2\pi \le  h_{{B}}([z, e^{i
\delta} z]) <  2 \pi.
\end{equation}
From (\ref{assu-20}),  we have
\begin{equation}\label{inc-8}
\langle q_{L_{0}+1 + 2N_{1}(m, \theta)} \theta \rangle  \cdot 2 \pi <
h_{{B}}([ e^{-i \delta} z, z]) < 2 \pi.
\end{equation}
We thus have (\ref{MM-1}) in this case  by taking
$$
M_{1}(m, \theta) = \min\{\frac{1}{\langle q_{L_{0}} \theta \rangle}, \frac{1}{
\langle q_{L_{0}+1 + 2N_{1}(m, \theta)} \theta \rangle}\}.
$$

Now assume that neither  of $\langle z, {B}^{q_{_{L_{0}}}}(z)
\rangle$ and  $\langle z, {B}^{q_{_{L_{0}+1}}}(z) \rangle$ is
contained  in $[e^{-i \delta} z, z]$ or  $[z, e^{i\delta}z]$. Let $k
\ge L_{0}+2$ be the least integer such that either $[e^{-i \delta}
z, z]$ or $[z, e^{i\delta}z]$ contains
 $\langle  z, {B}^{q_{k}}(z) \rangle$. Suppose that
\begin{equation}\label{inc-p}
\langle z,  {B}^{q_{k}}(z) \rangle = [z, {B}^{q_{k}}(z)]
\subset [z, e^{i \delta} z].
\end{equation}
The other cases can be treated in the same way. Then by the
assumption and the definition of $k$, we have
\begin{equation}\label{inc-1}
[z,  {B}^{q_{k}}(z)] \subset [z, e^{i \delta} z] \subset [z,
{B}^{q_{k-2}}(z)]
\end{equation}
and
\begin{equation}\label{inc-2}
 [e^{-i \delta} z, z] \subset
[{B}^{q_{k-1}}(z), z].
\end{equation}
Let $J(m, \theta)$ be the constant in Lemma~\ref{lemma:DF}. By
(\ref{inc-1}) and Lemma~\ref{lemma:DF}, it follows that
\begin{equation}\label{ii-1}
\big{|}[{B}^{q_{k-1}}(z), z]\big{|} \le J(m, \theta)
\big{|}[z, {B}^{q_{k}}(z)] \big{|} \le  J(m, \theta) \delta.
\end{equation}

Note that for $n \ge L_{0}$,
\begin{equation}\label{X-s}
\langle {B}^{q_{n+2}-q_{n+1}}(z), {B}^{q_{n+2}}(z) \rangle \cup \langle {B}^{q_{n+2}}(z), z \rangle \subset \langle {B}^{q_{n}}(z), z \rangle.
\end{equation}
  By the first assertion of Lemma~\ref{lemma:DF} we have
\begin{equation}\label{x-1}
|\langle {B}^{q_{n+2}-q_{n+1}}(z), {B}^{q_{n+2}}(z) \rangle| \ge J(m, \theta)^{-1} |\langle {B}^{q_{n+2}}(z), {B}^{q_{n+2}+q_{n+1}}(z) \rangle|.
\end{equation}
and
\begin{equation}\label{x-2}
 |\langle {B}^{2q_{n+2}}(z), {B}^{q_{n+2}}(z) \rangle| \ge   J(m, \theta)^{-1} |\langle {B}^{q_{n+2}}(z), z \rangle| .
\end{equation}

 By the second assertion of Lemma~\ref{lemma:DF}, we have
\begin{equation}\label{x-3}
 |\langle {B}^{q_{n+2}}(z), {B}^{q_{n+2}+q_{n+1}}(z) \rangle| \ge J(m, \theta)^{-1} |\langle {B}^{2q_{n+2}}(z), {B}^{q_{n+2}}(z) \rangle|.
\end{equation}

From (\ref{x-1})-(\ref{x-3}), we have
\begin{equation}\label{Y-s}
|\langle {B}^{q_{n+2}-q_{n+1}}(z), {B}^{q_{n+2}}(z) \rangle| > J(m, \theta)^{-3} |\langle {B}^{q_{n+2}}(z), z \rangle|.
\end{equation}
From (\ref{X-s}) and (\ref{Y-s}) we have
\begin{equation}\label{ii-2}
\big{|} \langle {B}^{q_{n}}(z), z \rangle \big{|} \ge (1 + J(m, \theta)^{-3}) \big{|} \langle {B}^{q_{n+2}}(z), z \rangle \big{|}
\end{equation}
holds for all $n \ge L_{0}$.  Let $N_{2}(m, \theta) > 0$ be the
least positive  integer such that
$$
\big{(}1 + J(m, \theta)^{-3}\big{)} ^{N_{2}(m, \theta)} > J(m, \theta).
$$
From (\ref{ii-1}) and (\ref{ii-2}), it follows that
\begin{equation}\label{inc-3}
[{B}^{q_{k-1 + 2N_{2}(m, \theta)}}(z), z] \subset [e^{-i
\delta}, z].
\end{equation}
From (\ref{inc-1}) we have
\begin{equation}\label{inc-7}
 \langle q_{k} \theta \rangle  \cdot 2\pi \le  h_{{B}}([z, e^{i
\delta} z]) \le \langle q_{k-2} \theta \rangle \cdot 2 \pi.
\end{equation}
From (\ref{inc-2}) and (\ref{inc-3}),  we have
\begin{equation}\label{inc-8}
\langle q_{k-1 + 2N_{2}(m, \theta)} \theta \rangle \cdot 2 \pi <
h_{{B}}([ e^{-i \delta} z, z]) <\langle q_{k-1}
\theta\rangle \cdot 2 \pi.
\end{equation}
 Since
$\theta$ is of bounded type, from (\ref{inc-7}) and (\ref{inc-8}),
it follows that there exists an $1 < M_{2}(m, \theta) < \infty$
depending only on $m$ and $\theta$ such that in this case
$$
M_{2}(m, \theta)^{-1} < \frac{\big{|}h_{{B}}([z, e^{i \delta}
z])\big{|}}{ \big{|}h_{{B}}([ e^{-i \delta} z, z])\big{|}}<
M_{2}(m, \theta).
$$

Theorem B then follows by taking $M(m, \theta) = \max\{M_{1}(m,
\theta), M_{2}(m, \theta)\}$.

%%%%%%%%%%%%%%%%%%%%%%%%%%%%%%%%%%%%%%%%%%%%%%%%%%%%%%%%%%%%%%%%%%%%%%
%%%%%%%%%%%%%%%%%%%%%%%%%%%%%%%%%%%%%%%%%%%%%%%%%%%%%%%%%%%%%%%%%%%%%%

%%%%%%%%%%%%%%%%%%%%%%%%%%%%%%%%%%%%%%%%%%%%%%%%%%%%%%%%%%%%%%%%%%%%%%
%%%%%%%%%%%%%%%%%%%%%%%%%%%%%%%%%%%%%%%%%%%%%%%%%%%%%%%%%%%%%%%%%%%%%%

\bibliographystyle{amsalpha}

\end{document}